\documentclass[12pt]{amsart}

\usepackage[english]{babel}
\usepackage[utf8]{inputenc}
\usepackage{times}
\usepackage[T1]{fontenc}
\usepackage{graphicx}
\usepackage{amscd,amsmath,amsfonts,amstext,amssymb,amsbsy,amsopn,amsthm,eucal}
\usepackage{txfonts}
\usepackage[figuresright]{rotating}
\usepackage{rotating,booktabs}
\usepackage{caption}
\usepackage{multirow}
\usepackage{tensor}
\usepackage[table]{xcolor}
\usepackage{fancyhdr}
\usepackage{color} 
\usepackage{hyperref}

\hypersetup{
  pdftitle={Warped product},
  pdfauthor={CLZ},
  pdfsubject={WP},
  pdfkeywords={IMCF, Penrose inequality, mass, RN-AdS}
  pdfpagelayout=SinglePage,
  pdfpagemode=UseOutlines,  pdfpagemode=UseOutlines,
  colorlinks,
  linkcolor=[rgb]{0,0,0.7},
  urlcolor=[rgb]{0,0,0.4},
  citecolor=[rgb]{0.4,0.1,0}
  }

\newcommand{\NE}{N_\epsilon}

\newcommand{\ps}[2]{\langle{#1},{#2}\rangle}

\newcommand{\n}{\nabla}
\newcommand{\bn}{\bar{\nabla}}

\newcommand{\hn}{\hat{\nabla}}

\newcommand{\barg}{\bar{g}}

\newcommand{\tg}{\tilde{g}}
\newcommand{\hg}{\hat{g}}
\newcommand{\hs}{\hat{\sigma}}

\def\pt{\partial}

\newcommand{\ld}{\lambda}

\newcommand{\RN}{\overline{R}}

\newcommand{\RicN}{Ric_{\barg}}

\DeclareMathOperator{\Ric}{Ric}

\def\dfrac{\displaystyle\frac}

\newtheorem{prop}{Proposition}[section]
\newtheorem{thm}[prop]{Theorem}

\newtheorem{lem}[prop]{Lemma}
\newtheorem{rem}[prop]{Remark}

\newtheorem{defn}[prop]{Definition}

\numberwithin{equation}{section}

\begin{document}


\baselineskip=17pt


\title[A Penrose type inequaltiy for graphs]{A Penrose type inequaltiy for graphs over Reissner-Nordstr\"om-anti-deSitter manifold }

\author[Daguang Chen, Haizhong Li, Tailong Zhou]{Daguang Chen, Haizhong Li, Tailong Zhou}

\thanks{*The first author was supported by  NSFC grant No.11471180 and the second author and the third author were supported by NSFC grant No.11671224.}


\maketitle


\begin{abstract}
	In this paper,  we use the inverse mean curvature flow to  establish  an optimal Minkowski type inquality,  weighted Alexandrov-Fenchel inequality for the mean convex star shaped hypersurfaces in Reissner-Nordstr\"om-anti-deSitter manifold and   Penrose type inequality for asymptotically locally hyperbolic manifolds in  which can be realized as graphs 	over Reissner-Nordstr\"om-anti-deSitter manifold.
\end{abstract}

\section{Introduction}

The famous positive mass conjecture in general relativity  states:
any asymptotically flat Riemannian manifold with a suitable decay order
and with nonnegative scalar curvature has the nonnegative ADM mass. Moreover, equality holds if and only if the manifold is isometric to the Euclidean space  with the standard metric.
The positive mass theorem first proved by Schoen and Yau \cite{SY} in 1979 using minimal surface techniques and then by Witten \cite{W} in 1981 using spinors. Recently, Schoen and Yau \cite{SY2017} proved the positive mass theorem in all
dimension.

The Penrose inequality in general relativity as refinement of the positive mass theorem states that the total mass of a spacetime is no less than the mass of its black holes. In the asymptotically
flat case, which corresponds to a vanishing cosmological constant,
the Riemannian Penrose inequality reads that
\begin{equation}\label{Penrose-ADM}
m_{ADM} \geq \frac 12\left(\frac{|\Sigma|}{\omega_{n-1}}\right)^{\frac{n-2}{n-1}},
\end{equation}
where $m_{ADM}$ is the ADM mass of the asymptotically flat Riemannian manifold with horizon and $|\Sigma|$ is the area of  $\Sigma$. The Riemannian Penrose inequality \eqref{Penrose-ADM} has been
established by Huisken–Ilmanen \cite{HI} by using inverse mean curvature flow    for a connected horizon and Bray \cite{Bray01} by using conformal flow for an arbitrary
horizon in dimension 3. Later, Bray’s approach was generalized to any dimension $n\leq 7$, as proven by Bray and Lee in \cite{BL}. For related results, see  the excellent surveys \cite{Bray} and \cite{Mars}.
Lam \cite{Lam} proved \eqref{Penrose-ADM} in all dimensions for an   asymptotically flat manifold which is a graph over $\mathbb{R}^n$. Mirandola and Vit\'orio \cite{MV} generalized Lam’s result to arbitrary codimension graph
with flat normal bundle. Recently, Ge et al. \cite{GWW1} introduced a new mass, which they called
Gauss-Bonnet-Chern mass, and they proved Penrose type inequalities in this case. Wei, Xiong and the second author \cite{LWX} obtained Penrose-type inequality of the second Gauss–Bonnet–Chern mass  for the graphic manifold with flat normal bundle.

In recent years, extending the previous results to a spacetime with a negative cosmological constant
attracts many authors’ interest. In the time symmetric case, $(M^n,g)$ is now an asymptotically  hyperbolic  manifold with an outermost minimal boundary $\Sigma$.
The notion of mass for this class of manifolds was first defined mathematically by
Wang \cite{Wang} and Chru\'sciel and Herzlich \cite{CH}.
For the  asymptotically hyperbolic manifolds, Chru\'sciel and Herzlich  in \cite{CH} introduced a mass-like invariant, which generalizes the ADM mass. See also \cite{CN,Herzlich1,Herzlich,M, ZhangX}. For  the   mass in  asymptotically hyperbolic manifolds, the corresponding Penrose conjecture states
\begin{equation}\label{Penrose-H}
m^{\mathbb H}\geq
\frac{1}{2}\left[
\left( \frac{|\Sigma|}{\omega_{n-1}} \right)^{\frac{n-2}{n-1}}
+ \left(\frac{|\Sigma|}{\omega_{n-1}}\right)^{\frac{n}{n-1}}
\right],
\end{equation}
if $R_g+n(n-1)\geq 0$, where $R_g$ is the scalar curvature of $g$.
A recent result by Neves \cite{Neves} shows that
it is not possible to prove \eqref{Penrose-H} only adapting Huisken and Ilmanen’s inverse mean curvature flow metheod in general.
In dimension $3$, Lee and Neves \cite{LN} were able to use the
inverse mean curvature flow to establish a Penrose type inequality for asymptotically locally
hyperbolic manifolds in the so-called negative mass range.
Dahl-Gicquaud-Sakovich \cite{DGS} and de Lima-Gir\~ao \cite{dLG3} proved  \eqref{Penrose-H}  for asymptotically hyperbolic graphs over the hyperoblic space $\mathbb {H}^n$ with the help of of the weighted hyperbolic Alexandrov-Fenchel inequality
\begin{equation} \label{AF-deLG}
\int_{\Sigma}  \cosh r\, H d\mu \geq (n-1)  \omega_{n-1}\left[\left(\frac{|\Sigma|}{\omega_{n-1}}\right)^{\frac{n-2}{n-1}}+\left(\frac{|\Sigma|}{\omega_{n-1}}\right)^{\frac{n}{n-1}}\right],
\end{equation}
provied $\Sigma$ is star-shaped and strictly mean-convex (i.e. $H>0$). It was proved by de Lima and Gir\~ao in \cite{dLG3}. The sharp Alexandrov–Fenchel-type
inequality  has an independent interest.  Recently, there have  been considerable progress in establishing
Alexandrov–Fenchel type inequalities, see for instance \cite{BHW,GWW3,GWWX,LWX,WX} and the references therein.


In this paper, we will consider the  Penrose inequality for asymptotically locally hyperbolic
graphic manifold over  Reissner-Nordtr\"{o}m-anti-deSitter (called "Reissner-Nordstr\"{o}m-AdS" for short) manifold. Firstly, we briefly recall the definition of Reissner-Nordstr\"{o}m-AdS manifold.
We fix three positive
numbers $m,q$ and $\kappa$, where $q<m, \kappa<<\infty$ such that the equation  $\epsilon+\kappa^2s^2-2ms^{2-n}+q^2s^{4-2n}=0$ has positive solutions, and let $s_0$ be the larger one.  Let $(N^{n-1}_{\epsilon}, \hg)$ be a closed space form of constant sectional curvature $\epsilon=0, \pm 1$ and $f=\sqrt{\epsilon+\kappa^2 s^2-2ms^{2-n}+q^2s^{4-2n}}$. The Reissner-Nordstr\"{o}m-anti-deSitter  manifold is a warped product manifold  $P=(s_0,\infty)\times N_{\epsilon}$ with the metric
$
\barg=\dfrac{1}{f^2}ds^2 + s^2 \hg,
$ where $\hg$ is the standard metric on $\NE$. The hypersurface $\partial P=\{s_{0}\}\times \NE$ is referred to  as the horizon.
Here we must remark that
the Reissner-Nordstr\"om-AdS manifold is referred to the manifold $P=(s_0,\infty)\times N_{\epsilon}$ with $\epsilon=1$ in \cite{WangZH}. In this case, Z.H. Wang \cite{WangZH} obtained a Minkowski type inequality by using the inverse mean curvature flow.
Motivated by \cite{WangZH}, we obtain

\begin{thm}\label{MainThm-Minkow}
Let $\Sigma$ be a compact mean convex, star-shaped and embedded hypersurface in Reissner-Nordstr\"{o}m-AdS manifold $P$ and let $\Omega$ denote the enclosed region by $\Sigma$ and the horizon $\partial P=\{s_{0}\}\times \NE$, then we have
	\begin{equation}\label{Minkow0}
	\begin{aligned}
	\int_{\Sigma} f H d\mu-n(n-1)\kappa^2 \int_{\Omega} f d v\geq&(n-1)\epsilon\vartheta_{n-1}
	\left(\left(\frac{|\Sigma|}{\vartheta_{n-1}} \right)^{\frac{n-2}{n-1}}-\left(\frac{|\partial P|}{\vartheta_{n-1}} \right)^{\frac{n-2}{n-1}}\right)\\
	&+(n-1)q^2\vartheta_{n-1}\left( \left(\frac{|\Sigma|}{\vartheta_{n-1}} \right)^{-\frac{n-2}{n-1}}-\left(\frac{|\partial P|}{\vartheta_{n-1}} \right)^{-\frac{n-2}{n-1}}\right),
	\end{aligned}	
	\end{equation}
	where $\vartheta_{n-1}=|\NE|$ and  $|\partial P|$ is the area of horizon $\{s_0\}\times \NE$. Equality in \eqref{Minkow0} holds if and only if $\Sigma$ is a slice $\{s\}\times \NE$ for $s\in [s_0,\infty)$.
\end{thm}
\begin{rem}
	For $\epsilon=1$ in Theorem \ref{MainThm-Minkow}, the inequality \eqref{Minkow0} is reduced to Theorem 1 in \cite{WangZH}. When $q=0$ and $\epsilon=1$, Theorem \ref{MainThm-Minkow} reduces to the Minkowski inequality proved in \cite{BHW}; when $q=0$ and $\epsilon=0,-1$,  Theorem \ref{MainThm-Minkow} was proved in \cite{GWWX}.
\end{rem}

\begin{thm}\label{AF-ineq}
	Let $\Sigma$ be a compact embedded hypersurface which is star-shaped with positive mean curvature in $P$, then we have
	\begin{equation}\label{AF-ineq1}
	\begin{aligned}
	\int_{\Sigma} f H d\mu\geq&
	(n-1)\kappa^2\vartheta_{n-1} \left(\left(\frac{|\Sigma|}{\vartheta_{n-1}} \right)^{\frac{n}{n-1}}-\left(\frac{|\partial P|}{\vartheta_{n-1}} \right)^{\frac{n}{n-1}} \right)\\
	&+(n-1)\epsilon\vartheta_{n-1}
	\left(\left(\frac{|\Sigma|}{\vartheta_{n-1}} \right)^{\frac{n-2}{n-1}}-\left(\frac{|\partial P|}{\vartheta_{n-1}} \right)^{\frac{n-2}{n-1}}\right)\\
	&+(n-1)q^2\vartheta_{n-1}\left( \left(\frac{|\Sigma|}{\vartheta_{n-1}} \right)^{-\frac{n-2}{n-1}}-\left(\frac{|\partial P|}{\vartheta_{n-1}} \right)^{-\frac{n-2}{n-1}}\right),
	\end{aligned}	
	\end{equation}
	where $\partial P=\{s_{0}\}\times N.$ Equality holds if and only if $\Sigma$ is a geodesic slice.
\end{thm}
\begin{rem}
If the parameter $q$ vanishes, the  Reissner-Nordstr\"om-AdS manifold becomes the Kottler space. Theorem \ref{AF-ineq} is reduced to Theorem 1.5 in \cite{GWWX}.
\end{rem}

Now we state that the  inequality for mass in  asymptotically hyperbolic graph manifolds over Reissner-Nordstr\"{o}m-AdS manifold $P$.

\begin{thm}\label{Penrose0}
	Suppose $M^{n}\subset Q$ is an ALH graph over $P$ with inner boundary $\Sigma$, associated to a smooth function $u:P\backslash\Omega\rightarrow\mathbb{R}$. Assume that $\Sigma$ is in a level set of $u$ and $|\bn u|\rightarrow\infty\ as\ x\rightarrow\Sigma$. Then
	\begin{equation}\label{Penrose0-ineq0}
	m(M,g)\geq m+c_n\left(\int_{M}\left(R_{g}-R_{\barg}\right)\left\langle\frac{\partial}{\partial t},\xi\right\rangle d\mu_{M}+\int_{\Sigma}fHd{\mu}\right),
	\end{equation}
	where $H$ is the mean curvature is of $\Sigma $ in $(P,\barg)$, $\xi$ is the unit outer normal of $(M,g)$ in $(Q,\tg)$,  the constant $c_n$ is defined by
	\begin{equation*}
	c_n=\frac{1}{2(n-1)\vartheta_{n-1}}
	\end{equation*}	
	and
	\begin{equation*}
	R_{\barg}=-n(n-1)\kappa^2+(n-1)(n-2)q^2 s^{2-2n}
	\end{equation*}
	is the scalar curvature of the  Reissner-Nordstr\"om-AdS manifold $P$. Equlity in \eqref{Penrose0-ineq0} holds if and only if $M$ is rotationaly symmetric. Moreover, if $R_{g}\geq \bar{R}$, we have
	\begin{equation}\label{Penrose0-ineq1}
	m(M,g)\geq m+c_n\int_{\Sigma}fHd{\mu}.
	\end{equation}
\end{thm}

From Theorem \ref{AF-ineq} and \eqref{Penrose0-ineq1}, we immediately deduce the Penrose inequality for ALH graphs.
\begin{thm}\label{Penrose}
	If $M\subset Q$ is a ALH graph as in Theorem \ref{Penrose0}, such that $\Sigma\subset(P,\barg)$ is star-shaped and mean-convex, moreover, if in additon $R_{g}\geq \bar{R}$, we have
	\begin{equation}\label{Penrose-ineq1}
	m(M,g)\geq\frac{1}{2}\left(\kappa^2\left(\frac{|\Sigma|}{\vartheta_{n-1}} \right)^{\frac{n}{n-1}}+\epsilon\left(\frac{|\Sigma|}{\vartheta_{n-1}} \right)^{\frac{n-2}{n-1}}+q^2\left(\frac{|\Sigma|}{\vartheta_{n-1}} \right)^{-\frac{n-2}{n-1}}\right).
	\end{equation}
	Equality in \eqref{Penrose-ineq1} is achieved by the  Reissner-Nordstr\"om-AdS manifold $P$.
\end{thm}
\begin{rem}\label{Rigidity}
		It's easy to show that the  Reissner-Nordstr\"om-AdS manifolds $(P,\barg)$ with constant $\epsilon,\ \kappa,\ m,\ q$ and $-\kappa^{2}> c_{\epsilon}$ (see Remark \ref{positive root}) can be represented as an ALH graph over another  Reissner-Nordstr\"om-AdS manifold with constant $\epsilon,\ \kappa,\ m',\ q$ ,$m'<m$. From \eqref{mass-1} and \eqref{horizon mass} we konw that equality is achieved by the  Reissner-Nordstr\"om-AdS manifolds.
From the argument of Huang-Wu \cite{HuangWu}, we believe that the rigidity in Theorem \ref{Penrose} should hold, i.e., the eqaulity holds in  the Penrose type inequality \eqref{Penrose-ineq1}  if and only if $M$ is exactly the  Reissner-Nordstr\"om-AdS manifold.
\end{rem}

\begin{rem}
A similar argument appeared in \cite{HI08}, \cite{GL} and \cite{LW} implies that our main results in Theorem \ref{MainThm-Minkow}, Theorem \ref{AF-ineq} and Theorem \ref{Penrose} also hold for star-shaped and weakly mean convex hypersurface $\Sigma$ in  Reissner-Nordstr\"om-AdS manifold $(P,\barg)$.
\end{rem}

\begin{rem}
By use of the weak solution of inverse mean curvature flow  (see \cite {HI}) and the techniques of \cite {Wei}, it is an interesting problem to prove the Minkowski-type inequality for outward minimizing hypersurfaces in  Reissner-Nordstr\"om-AdS manifold $(P,\barg)$.
\end{rem}

The paper is organized as follows. In Section \ref{Prelim}, we collect some facts about the  Reissner-Nordstr\"om-AdS,  star-shaped hypersurfaces, inverse mean curvature flow and the asymptotically locally hyperbolic manifold;  In Section \ref{Existence}, we establish the long-time existence and convergence result of the inverse mean curvature flow for star-shaped and strictly mean convex hypersurface in  Reissner-Nordstr\"om-AdS manifold. In Section \ref{Sec:Proof of Theorem 1.1} and Section \ref{Sec:Proof of Theorem 1.3}, we will prove Theorem \ref{MainThm-Minkow} and Theorem \ref{AF-ineq}, repectively. In the last section, we will give the proof of Theorem \ref{Penrose0} and Theorem \ref{Penrose}.

\section{Preliminaries}\label{Prelim}
In this section, we collect some facts about the  Reissner-Nordstr\"om-AdS manifold, star-shaped hypersurfaces,  inverse mean curvature flow and asymptotic locally hyperoblic (ALH) mass of graphs.

\subsection{ Reissner-Nordstr\"om-AdS manifold}
We fix three positive
numbers $m,q$ and $\kappa$ such that the equation  $\epsilon+\kappa^2s^2-2ms^{2-n}+q^2s^{4-2n}=0$ has positive solutions, and let $s_0$ be the largest one.  Let $(N^{n-1}_{\epsilon}, \hg)$ be a closed space form of constant sectional curvature $\epsilon=0, \pm 1$.
As in \cite{WangZH,KI}, the  Reissner-Nordstr\"om-AdS manifold is defined by $P=(s_0,\infty)\times N_{\epsilon}$ with the metric
\begin{equation*}
\barg=\dfrac{1}{\epsilon+\kappa^2 s^2-2ms^{2-n}+q^2s^{4-2n}}ds^2 + s^2 \hg.
\end{equation*}
By a change of variable, the metric of  Reissner-Nordstr\"om-AdS manifold can be  rewritten as
\begin{equation}\label{metric-RN}
\barg=dr^2+\ld(r)^2\hg,
\end{equation}
where $\lambda(r)$ satisfies the ODE
\begin{equation}\label{lambda-ode}
\lambda'(r)=\sqrt{\epsilon+\kappa^2 \lambda^2-2m\lambda^{2-n}+q^2\lambda^{4-2n}}.
\end{equation}
Defining $f(\lambda)=\lambda'$, one can check that
\begin{equation}\label{deriative-para}
\begin{aligned}
f^\prime\lambda'=\ld^{\prime\prime}=&\kappa^2 \ld+m(n-2)\ld^{1-n}-(n-2)q^2\ld^{3-2n}.
\end{aligned}
\end{equation}
Then the function $f$ satisfies (see (1.2) in \cite{WangZH})
\begin{equation}\label{Static-Equ}
\begin{aligned}
(\bar{\Delta} f)\bar{g}-\bar{\n}^2 f+ f\RicN
=&(n-1)(n-2) q^2
\ld(r)^{4-2n}f\,
\hg,
\end{aligned}
\end{equation}
where $\RicN$ is the Ricci curvature tensor, $\bar{\Delta}$ and $\bn^2$ are the Laplacian and Hessian operators with repect to the metric $\barg$ of  Reissner-Nordstr\"om-AdS manifold.
In general, a Riemannian metric is called sub-static if $(\bar{\Delta} h)\bar{g}-\bar{\n}^2 h+ h\RicN\geq 0$ for some positive function $h$.
\begin{rem}
When $q\rightarrow0$, the  Reissner-Nordstr\"om-AdS manifold reduces to Kottler manifold (see \cite{GWWX,LW}).
\end{rem}
\begin{rem}\label{positive root}
	When $q>0$, to ensure $\psi(s):=\epsilon+\kappa^{2}s^{2}-2ms^{2-n}+q^{2}s^{4-2n}=0$ to have positive root we need the following conditions (see \cite{KI}). One can check $\psi(+\infty)=+\infty$, $\psi(0^{+})=+\infty$ and $\psi'(s)=0$ always have a single positive solution $s=a>0$. Then positive root for $\psi(s)=0$ exists if and only if $\psi(a)\leq 0$. From $\psi'(a)=0$, we get
	\begin{equation*}
	-\kappa^{2}=\frac{n-2}{a^{2}}\left(\frac{m}{a^{n-2}}-\frac{q^{2}}{a^{2n-4}}\right)\ \Big(=:g(a)\Big).
	\end{equation*}
	Combining it with $\psi(a)\leq 0$ we have
	\begin{equation*}
	0\geq\epsilon-\frac{nm}{a^{n-2}}+\frac{(n-1)q^{2}}{a^{2n-4}}.
	\end{equation*}
	If $D:=n^{2}m^{2}-4\epsilon(n-1)q^{2}\geq 0$, then
	\begin{equation*}\label{}
	\frac{1}{a^{n-2}}\leq\frac{1}{a_{c}^{n-2}}:=\frac{nm+\sqrt{D}}{2(n-1)q^{2}}
	\end{equation*}
	ensures $\psi(a)\leq 0$. Equivalently,
	\begin{equation}\label{positive-root}
	-\kappa^{2}=g(a)\geq g(a_{c})=
	\begin{cases}
	&c_{0}:=-\frac{(n-1)^{2}-1}{(n-1)^{2}}\left(\frac{n}{n-1}\right)^{\frac{2}{n-2}}\left(\frac{m^{2}}{q^{2}}\right)^{\frac{n}{n-2}}m^{-\frac{2}{n-2}},\ \epsilon=0,\\
	&c_{-1}:=-2^{\frac{2}{n-2}}\frac{n-2}{n-1}\left(\sqrt{D}-(n-2)m\right)\left(\sqrt{D}-nm\right)^{-\frac{n}{n-2}},\ \epsilon=-1,\\
	&c_{1}:=-2^{\frac{2}{n-2}}\frac{n-2}{n-1}\left(\sqrt{D}-(n-2)m\right)\left(nm-\sqrt{D}\right)^{-\frac{n}{n-2}},\ \epsilon=1.
	\end{cases}
	\end{equation}
When $\epsilon=0,\ -1$, $D\geq 0$ automatically. When $\epsilon=1$, $-\kappa^{2}\geq c_{1}$ implies $q^{2}<m^{2}$. Therefore, $\psi(s)=0$ has positive root when $-\kappa^{2}\geq c_{\epsilon}$ ($c_{\epsilon}=c(\epsilon,n,m,q)$ as above) and $q^{2}<m^{2}$ if $\epsilon=1$. When $\epsilon=1$, \eqref{positive-root} can be  written as $q<m,\ \kappa<<\infty$, which was observed by Z.H. Wang in \cite{WangZH}.
\end{rem}
\begin{rem}
One can check that when $\lambda\geq s_0$, where $s_0$ is the largest positive root of the equation $\epsilon+\kappa^2s^2-2ms^{2-n}+q^2s^{4-2n}=0$, then $\lambda''\geq0$.
\end{rem}

\subsection{Asymptotic behaviors of the  Reissner-Nordstr\"om-AdS manifold}$\ $\\
Defining
$$
r(s)=\int_{s_{0}}^{s}\frac{dt}{\sqrt{\kappa^{2}t^{2}+\epsilon}}-\int_{s}^{\infty}\left(\frac{1}{\sqrt{\epsilon+\kappa^2 t^2-2mt^{2-n}+q^2t^{4-2n}}}-\frac{1}{\sqrt{\kappa^{2}t^{2}+\epsilon}}\right)dt,
$$
 we have $r(s)$ is the inverse function of $\lambda(r)$ up to a constant from \eqref{lambda-ode}.\\
For $\epsilon=1$, in \cite{WangZH}, Z.H. Wang obtained
	$$
	r(s)=\kappa^{-1}\sinh(\kappa r)-\frac{m}{n}\kappa^{-3}s^{-n}+\kappa^{-4}O(s^{-n-2}).
	$$
	By Taylor expansion,
	$$
	\begin{aligned}
	\sinh(\kappa r(s))&=\kappa s-\frac{m}{n}\kappa^{-1}s^{1-n}+\kappa^{-3}O(s^{-n-1})\\
	&=\kappa s-\frac{m}{n}\kappa^{-1}(\kappa^{-1}\sinh(\kappa r))^{1-n}+O((\kappa^{-1}\sinh(\kappa r))^{-1-n}).
	\end{aligned}
	$$
	Then
	$$
	\begin{aligned}
	\lambda(r)&=\kappa^{-1}\sinh(\kappa r)+\frac{m}{n}\kappa^{n-3}\sinh^{-n+1}(\kappa r)+O(\sinh^{-n-1}(kr))\\
	&=O(e^{\kappa r}).
	\end{aligned}
	$$
For $\epsilon=-1$ and $\epsilon=0$, by similar calcluations, we have
	\begin{align}
	\lambda(r)=&\kappa^{-1}\cosh(\kappa r)+\frac{m}{n}\kappa^{n-3}\cosh^{-n+1}(\kappa r)+O(\cosh^{-n-1}(kr))
	=O(e^{\kappa r}),\ \ \ \ \ \text{for}\ \  \epsilon=-1,\nonumber\\
	\lambda(r)=&e^{\kappa r}+\frac{m}{n}\kappa ^{-2}e^{-(n-1)\kappa r}+O(e^{-(n+1)\kappa r})
	=O(e^{\kappa r}),\ \ \ \ \ \text{for}\ \  \epsilon=0.\nonumber
	\end{align}
Therefore, the function $\lambda(r)$ has the following asymptotic expansion
\begin{equation}\label{Asy-r}
\lambda(r)=O(e^{\kappa r}).
\end{equation}
Let $\bn$ be the covariant derivative of $P$. The Riemann curvature tensor of $P$ is given by
\begin{equation}
\RN(X,Y)Z=\bn_{Y}\bn_{X}Z-\bn_{X}\bn_{Y}Z+\bn_{[X,Y]}Z.
\end{equation}
Let $\{e_{\alpha}\}_{\alpha=0}^{n-1}$ be an orthonormal frame and let $\RN_{\alpha\beta\gamma\mu}=\barg\Big(\RN(e_{\alpha},e_{\beta})e_{\gamma},e_{\mu}\Big)$ be the Riemannian curvature tensor  and   $\bn$ be the covariant derivative of the  Reissner-Nordstr\"om-AdS manifold $(P,\barg)$, respectively.
The asymptotic expansion of Riemannian curvature $\RN_{\alpha\beta\gamma\mu}$ and Ricci curvature $\RicN$ are given by (see Lemma 10 in \cite{WangZH})
\begin{align}
	\RN_{\alpha\beta\gamma\mu}=&-\kappa^{2}(\delta_{\alpha\gamma}\delta_{\beta\mu}-\delta_{\alpha\mu}\delta_{\beta\gamma})+O(e^{-n\kappa r}),\label{Asy-curv}\\
	\bn_{\rho}\RN_{\alpha\beta\gamma\mu}=&O(e^{-n\kappa r}),\label{Asy-Dcurv}\\
	\RicN(\partial_{r},\partial_{r})=&-(n-1)\kappa^{2}+O(e^{-n\kappa r}),\label{Asy-Ricurv1}\\
	\ld^{-2}\RicN(\partial_{i},\partial_{j})=&-(n-1)\kappa^{2}\hg_{ ij}+O(e^{-n\kappa r})\label{Asy-Ricurv2}.
\end{align}

\subsection{Star-shaped hypersurface in the  Reissner-Nordstr\"om-AdS manifold}
In this subseciton, we recall some basic properties of star-shaped hypersurfaces in $(P,\barg)$ (see in \cite{Ge-2006,Gerhardt,LW}). Let $\theta=\{\theta^{i}\}_{i=1,\cdots,n-1}$ be a local coordinate systems on $N_{\epsilon}$ and $\partial_{r}$ be the radial vector fields. If $\Sigma$ is a smooth closed hypersurface in the  Reissner-Nordstr\"om-AdS manifold $(P,\barg)$, the hypersurface $\Sigma$ is called star-shaped if the support function $\ps{\lambda\pt_r}{\nu}>0$ on $\Sigma$, which implies that $\Sigma$ could be parameterized by a graph
\begin{equation*}
\Sigma=\Big
\{(r(\theta),\theta):\theta\in N_{\epsilon}\Big\}.
\end{equation*}
As in \cite{BHW,Di,Ge-2006,LW}, we define a function $\varphi$ on $N_{\epsilon}$ by $\varphi(\theta)=\Phi(r(\theta))$, where $\Phi(r)$ is a positive function satisfying $\Phi'(r)=1/{\lambda(r)}$.
The tangential vector field along $\Sigma$ can be expressed in the form
\begin{equation*}
X_{i}=\partial_{i}+r_{i}\partial_{r}=\partial_{i}+\lambda\varphi_{i}\partial_{r}
\end{equation*}
and then the induced metric on $\Sigma$ is given by
\begin{equation}
g_{ij}=\lambda^{2}(\hg_{ij}+\varphi_{i}\varphi_{j}).
\end{equation}
The unit outward normal vector field $\nu$ of the hypersurface $\Sigma$ could be written as
\begin{equation}\label{normal}
\nu=\frac1{\upsilon}\left(\partial_{r}-\frac{r^{i}}{\lambda^{2}}\partial_{i}\right),\ \upsilon=\sqrt{1+|\hn\varphi|^{2}_{\hg}},
\end{equation}
where $r^{i}=r_{j}\hg^{ij}$, $\Big(\hg^{ij}\Big)=\Big(\hg_{ij}\Big)^{-1}$ and $\hn$ denotes the Levi-Civita connection on $N_{\epsilon}$.

Denoting $h_{ij}$ by the components of the second fundamental form of the hypersurce $\Sigma$. Then we have
\begin{equation}\label{II}
h_{ij}=\frac{\lambda}{\upsilon}\left(\lambda'(\hg_{ij}+\varphi_{i}\varphi_{j})-\varphi_{ij}\right),\
h_{i}^{j}=\frac{1}{\upsilon\lambda}\left(\lambda'\delta_i^j-\hat{\sigma}^{jk}\varphi_{ki}\right),
\end{equation}
where $\varphi_{ki}=\hn_i\hn_k\varphi$, $g^{ij}=\lambda^{-2}\hat{\sigma}^{jk}$ and  $\hat{\sigma}^{jk}:=\hg^{jk}-\frac{\varphi^j\varphi^k}{\upsilon^2}$ with $\varphi^j=\hg^{jk}\varphi_k$.
The mean  curvature is given by
\begin{equation}\label{mean-curv}
H=g^{ij}h_{ij}=\frac{1}{\ld\upsilon}\left((n-1)\ld^\prime-\hat{\sigma}^{ij}\varphi_{ij}\right).
\end{equation}

\subsection{The inverse mean curvature flow in  Reissner-Nordstr\"om-AdS manifold}
Let $\Sigma_{0}$ be a smooth, strictly mean-convex, star-shaped closed embedded hypersurface in $(P,\barg)$. The inverse mean curvature flow in the  Reissner-Nordstr\"om-AdS manifold is a family of embeddings $X:\Sigma\times [0,T)\rightarrow(P, \barg)$ satisfying
\begin{equation}\label{IMCF}
\left\{
\begin{aligned}
\frac{\partial}{\partial t}X=&\frac{1}{H}\nu,\\
X_{0}=&\Sigma_{0},
\end{aligned}
\right.
\end{equation}
where $\nu$ is the unit outer normal to hypersurface  $\Sigma_{t}=X(t,\Sigma)$ and  $H$ is the mean curvature of $\Sigma_t$. If the initial hypersurface is strictly mean convex, the short time existence result (see, e.g.,\cite{Ge-2006}) of \eqref{IMCF} implies the flow exists on a maximum time interval $[0,T)$. Thus it remains to study the long time behavior of the flow \eqref{IMCF}.

The equation \eqref{IMCF} is called the parametric form of the inverse mean curvature flow. For a graphic hypersurface,  the equation \eqref{IMCF} can be expressed in another form. Let $\Sigma_{0}=\{(r_{0}(\theta),\theta):\theta\in N_{\epsilon}\}$ be a star shaped hypersurface in  Reissner-Nordstr\"om-AdS manifold $(P,\barg)$ defined on $N_{\epsilon}$. Assume that  $\Sigma_{t}$ is the solution hypersurface  and  star shaped for  \eqref{IMCF}. Then it can be parametrized as
$$
\Sigma_{t}=\{(r(\theta,t),\theta):\theta\in N_{\epsilon}\}.
$$
As long as the solution of \eqref{IMCF} exists and remains star shaped, it is  equivalent to the following  non-parametric form of the flow (cf.\cite{BHW,Di,Gerhardt,GWWX})
\begin{equation}\label{Nonpara-flow}
\frac{\partial}{\partial {t}}r(\theta,t)=\frac{v}{H}.
\end{equation}
We have the following evolution equations (see \cite{BHW,Ge-2006,WangZH,Z}).
\begin{prop}[Evolution equations]\label{evl-equations}
Along the inverse mean curvature  flow \eqref{IMCF}, we have the following evolution equations:
\begin{align}
	\frac{\partial}{\partial t}g_{ij}=&\frac{2}{H}h_{ij},\label{evl-metric}\\
	 \frac{\partial}{\partial t}g^{ij}=	&-\frac{2}{H}h^{ij}\label{evl-Invmetric}\\
\frac{\partial}{\partial t}\nu=&\frac{1}{H^{2}}\n H,\label{evl-nu}\\
	\frac{\partial}{\partial t}d\mu=&d\mu\label{evl-measure},\\
	\frac{\partial }{\partial t}\varphi=&\frac{\upsilon}{\ld H}=\frac{\upsilon^2}{(n-1)\ld^\prime-\hat{\sigma}^{ij}\varphi_{ij}},\label{evl-varphi}
	\end{align}

	\begin{align}
	\frac{\partial}{\partial t}h_{ij}=&\frac{1}{H^{2}}\Delta h_{ij}-\frac{2}{H^{3}}H_{i}H_{j}-\frac{2}{H}{\RN}_{\nu i\nu j}+\frac{1}{H^{2}}\left(|A|^{2}+{\RN}^{k}_{~\nu k\nu}\right)h_{ij}\label{evl-sff}\\
	&-\frac{1}{H^{2}}\left(2h^{p}_{l}{\RN}^{l}_{~ijp}+h_{j}^{p}{\bar{R}}^{k}_{~ikp}+h_{i}^{p}{\RN}^{k}_{~jkp}\right)-\frac{1}{H^{2}}\left(\bn_{k}{\RN}_{\nu ji}^{~~~k}+\bn_{i}{\RN}^{k}_{~\nu kj}\right),\nonumber\\
	\frac{\partial}{\partial t}H=	&\frac{1}{H^{2}}\Delta H-\frac{2}{H^{3}}|\n H|^{2}-\frac{|A|^{2}}{H}-\frac{\RicN(\nu,\nu)}{H},\label{evl-mean}
	\end{align}
where $\hat{\sigma}^{ij}=\hg^{ij}-\frac{\varphi^i\varphi^j}{\upsilon^2}$, $\RN_{\alpha\beta\gamma\rho}$ is the Riemannian curvature tensor and $Ric_{\barg}$ is the Ricci tensor  of $(P,\barg)$.
In the non-parametric form the function $u:=\frac{1}{H\ps{\lambda\partial r}{\nu}}=\frac{\upsilon}{\ld H}$  evolves under
	\begin{equation}\label{evl-u}
	\frac{\partial}{\partial t}u=\frac{\Delta u}{H^{2}}-2\frac{|\n u|^{2}}{uH^{2}}-2\frac{\ps{\n u}{\n H}}{H^{3}}-(n-1)\frac{\lambda''u}{\lambda H^{2}}.
	\end{equation}
\end{prop}

\subsection{Asymptotically locally hyperbolic manifold}
In order to define the mass of asymptotically locally hyperbolic manifold, we recall from (\cite{CH,dLG3,GWWX}) the following definition.
Fix $\epsilon=0,\pm1$, $\kappa>0$ and suppose $(N^{n-1}_{\epsilon},\hg)$ is a closed space form of sectional curvature $\epsilon$. Consider the product manifold $P_{\epsilon}=I_{\epsilon}\ \times\ N_{\epsilon}$, where $I_{-1}=(\frac{1}{\kappa},\infty)$ and $I_{0}=I_{1}=(0,\infty)$ endowed with the warped product metric
\begin{equation}\label{ALH-metric}
b_{\epsilon}=\frac{ds^{2}}{V_{\epsilon}^{2}(s)}+s^{2}\hg,\ s\in I_{\epsilon},\ \text{and}\ V_{\epsilon}(s)=\sqrt{\kappa^2 s^{2}+\epsilon}.
\end{equation}
We recall the following definition of asymptotic locally hyperbolic manifold and its mass which is a geometric invariant (see in \cite{CH}).
\begin{defn}\label{ALH-mfd}
	A Riemannian manifold $(M^{n},g)$ is called a asymptotically locally hyperbolic (ALH) if there exists a compact subset $K$ and a diffeomorphism at infinity $\Phi :M\backslash K\ \rightarrow N\ \times\ (s_{0},\infty)$, such that
	\begin{equation}
	||(\Phi^{-1})*g-b_{\epsilon}||_{b_{\epsilon}}+||\triangledown^{b_{\epsilon}}((\Phi^{-1})*g)||_{b_{\epsilon}}=O(s^{-\tau}),\ \ \tau>\frac{n}{2},
	\end{equation}
	and
	\begin{equation}
	\int_{M}V_{\epsilon}|R_{g}+\kappa^{2} n(n-1)|d\mu_{M}<\infty.
	\end{equation}
	Then the ALH mass can be defined as
	\begin{equation}
	m(M,g)=\frac{1}{2(n-1)\vartheta_{n-1}}\lim_{s\rightarrow\infty}\int_{N_{s}}\left(V_{\epsilon}(div^{b_{\epsilon}}e-{\rm d}\ tr^{b_{\epsilon}}e)+(tr^{b_{\epsilon}}e){\rm d}V_{b_{\epsilon}}-e(\n^{b_{\epsilon}}V_{b_{\epsilon}},\cdot)\right)\nu d\mu
	\end{equation}
	where $e:=(\Phi^{-1})*g-b_{\epsilon}$,$N_{s}=\{s\}\times \NE$, $\nu$ is the outer normal of $N_{s}$ induced by $b_{\epsilon}$ and $d\mu$ is the area element with respect to the induced metric on $N_{s}$.
\end{defn}
The corresponding  Reissner-Nordstr\"om-AdS manifold spacetime in general relavity is
\begin{equation*}
-f^{2}dt^{2}+\barg=-f^{2}dt^{2}+\frac{1}{f^2}ds^2+s^2\hg	
\end{equation*}
where $m$ is the mass and $q$ is the charge. Now we consider its Riemannian version, namely $Q=\mathbb{ R}\times P$ with mertic
\begin{equation}
\tg=f^{2}dt^{2}+\barg=f^{2}dt^{2}+\frac{1}{f^2}ds^2+s^2\hg	.
\end{equation}
We identify $P$ with the slice $\{0\}\times P\subset Q$ and consider a graph over $P$ or over $P\backslash\Omega$ in $Q$, where $\Omega$ is a compact smooth subset containing $\{0\}\times\partial P$. A graph associated to a smooth function $u:P\backslash\Omega\rightarrow\mathbb{ R}$ is a manifold $M^{n}$ with induced metric from $(Q,\tg)$,i.e.
\begin{equation}
g=f^2(s)\bn u\otimes\bn u+\barg
\end{equation}
where $\bn$ is the covariant derivative with respect to metric $\barg$.
\begin{defn}\label{ALH graph metric}
	We call that $M^{n}\subset Q$ is an ALH graph over $P\backslash\Omega$ (associated to a smooth function $u:P\backslash\Omega\rightarrow\mathbb{ R}$) if there exists a compact subset K and a diffeomorphism at infinity $\Phi:M\backslash K\rightarrow \NE\times (s_{0},+\infty)\subset P\backslash\Omega$, such that
	\begin{equation}
	||(\Phi^{-1})*g-\barg||_{\barg}+||\bn (\Phi^{-1})*g||_{\barg}=O(s^{-\tau}),\qquad \tau>\frac{n}{2}
	\end{equation}
	or equivalently,
	\begin{equation}
	|f\bn u|_{\barg}+|f\bn^{2}u+\bn f \otimes
	\bn u|_{\barg}=O(s^{-\tau}),\qquad
	 \tau>\frac{n}{2}
	\end{equation}
	and
	\begin{equation}
	\int_{M}f|R_{g}-R_{\barg}|d\mu_{M}<\infty.
	\end{equation}
\end{defn}
For such a graph over $P\backslash\Omega$, we can check definition \ref{ALH-mfd} is equivalent to definition \ref{ALH graph metric}, i.e. any ALH graph is an ALH manifold (see \cite{CH,GWWX}).
\section{Inverse mean curvature flows in the  Reissner-Nordstr\"om-AdS manifold}\label{Existence}

In this section we establish a long time existence and asymptotic behavior of the inverse mean curvature flow. Our proofs use the similar arguments in \cite{BHW,GWWX,LW,WangZH}.

\subsection{ $C^{0}$ and $C^1$-estimates}

By the parabolic maximum principle, we can obtian   $C^{0}$ and $C^1$-estimates of the inverse mean curvature flow \eqref{IMCF} in  Reissner-Nordstr\"om-AdS manifold.
\begin{lem} \label{C0-est}
	Let $\underline{r}(t)=\inf_N r(\theta,t)$ and $\bar{r}(t)=\sup_N r(\theta,t)$.
	The solution $r$ of \eqref{Nonpara-flow} satisfies
	\begin{equation}\label{C0-est1}
	\lambda(\underline{r}(0))e^{\frac{1}{n-1}t}\leq\lambda(r(\theta,t))\leq\lambda(\bar{r}(0))e^{\frac{1}{n-1}t},
	\end{equation}
	which imply $e^{\kappa r}=O(e^{\frac{1}{n-1}t})$.
\end{lem}
\noindent Next, we will give the lower and upper bounds of the mean curvature.
\begin{lem}\label{bound-mean}
	Along the flow \eqref{IMCF}, we have $H\leq(n-1)\kappa+O(e^{-\frac{2}{n-1}t})$, and $H\geq Ce^{-\frac{1}{n-1}t}$ for some constant $C$ depending only on the initial hypersurface $\Sigma_{0}$.
\end{lem}
\begin{proof}
From  \eqref{evl-mean}, \eqref{Asy-Ricurv1} and \eqref{Asy-Ricurv2}, we have
\begin{equation*}
	\begin{aligned}
	\frac{\partial}{\partial t}H^{2}&=\frac{1}{H^{2}}\Delta H^{2}-\frac{3}{2H^{2}}|\n H^{2}|^{2}-2|A|^{2}-2\RicN(\nu,\nu)\\
	&\leq\frac{1}{H^{2}}\Delta H^{2}-\frac{3}{2H^{2}}|\n H^{2}|^{2}-\frac{2}{n-1}H^{2}+2\kappa^{2}(n-1)+O(e^{-n\kappa r})
	\end{aligned}
\end{equation*}
	Using maximum principle and the estimate of $\lambda$ in Lemma \ref{C0-est}, we can deduce
	\begin{equation}
	\frac{d}{dt}H^{2}_{max}\leq-\frac{2}{n-1}H^{2}_{max}+2(n-1)\kappa^{2}+O({e^{\frac{n}{n-1}t}})
	\end{equation}
which implies  $H^{2}\leq(n-1)^{2}\kappa^{2}+O(e^{-\frac{2t}{n-1}})$.\\
Denoting
\begin{equation}\label{defn-G}
\frac{1}{G}=\frac{\upsilon}{\lambda H}=\frac{\upsilon^{2}}{(n-1)\lambda'-\hat{\sigma}^{ij}\varphi_{ij}},
\end{equation}
 where
 \begin{equation*}
   \upsilon^{2}=1+\varphi_{i}\varphi^{i},\quad\text{and}\quad \upsilon\lambda h_{i}^{j}=\lambda'\delta_{i}^{j}-\hat{\sigma}^{jk}\varphi_{ik}.
 \end{equation*}
Since $\lambda'$ can be seen as a function of $\varphi$, therefore  $G$ can  be considered as a function of $\varphi,\hn\varphi$ and $\hn^{2}\varphi$.
 Taking derivative   \eqref{evl-varphi} with $t$, we get
	\begin{equation}\label{evl-var2}
	\frac{\partial}{\partial t}\varphi_{t}=-\frac{1}{G^{2}}G^{ij}\left(\varphi_{t}\right)_{ij}-\frac{1}{G^{2}}G^{i}(\varphi_{t})_{i}-\frac{n-1}{G^{2}\upsilon^{2}}\lambda''\lambda\varphi_{t},
	\end{equation}
	where $G^{ij}=\frac{\partial G}{\partial\varphi_{ij}}=-\frac{1}{\upsilon^{2}}\hat{\sigma}^{ij}<0$ and $G^{i}=\frac{\partial G}{\partial\varphi_{i}}$.
	One can check that $\lambda''\geq 0$  and $\varphi_{t}>0$ according to it's expanding. By using maximum priciple, we have
	\begin{equation}\label{var-t}
	\varphi_{t}\leq C.
	\end{equation}
From \eqref{evl-varphi}, \eqref{var-t} and Lemma \ref{C0-est}, we obtain
\begin{equation*}
H\geq C\frac{\upsilon}{\lambda}\geq Ce^{-\frac{1}{n-1}t}.
	\end{equation*}
\end{proof}

\begin{lem}\label{bound-c1norm}
	We have $|\hn\varphi|_{\hg}=O\left(e^{-\frac{1}{(n-1)^{2}}t}\right)$.
\end{lem}
\begin{proof}
 Defining $w=\frac{1}{2}|\hn\varphi|^{2}_{\hg}$, then
\begin{equation}\label{evl-w}
\begin{aligned}
	\frac{\partial}{\partial t}w=&\frac{\partial}{\partial t}(\varphi_{k})\varphi^{k}
	=\hn_{k}\left(\frac{\upsilon}{\lambda H}\right)\varphi^{k}\\
=&\frac{1}{\lambda^{2}H^{2}}\left(\hat{\sigma}^{ij}w_{ij}-\upsilon^{2}G^{l}w_{l}-2(n-1)\lambda\lambda''w-2\epsilon(n-2)w-\hat{\sigma}^{ij}\varphi_{ki}\varphi_{j}^{k}\right)
\end{aligned}
\end{equation}
where $G$ is defined by \eqref{defn-G}.\\
Similar in \cite{GWWX}, by using $-\epsilon\leq\kappa^{2}\lambda^{2}-2m\lambda^{2-n}+q^{2}\lambda^{4-2n}$, \eqref{evl-w}, Lemma \ref{C0-est} and Lemma \ref{bound-mean},
we have
	\begin{equation*}
	\begin{aligned}
	\frac{d}{dt}w_{max}\leq&\frac{1}{H^{2}}\left(-2(n-1)\frac{\lambda''}{\lambda}w_{max}+2(n-2)(\kappa^{2}-2m\lambda^{-n}+q^{2}\lambda^{2-2n})w_{max}\right)w_{max}\\
	=&\frac{1}{H^{2}}\left(-2\kappa^{2}-2(n-2)(n+1)m\lambda^{-n}+2n(n-2)q^{2}\lambda^{2-2n}\right)w_{max}\\
	\leq&\left(-\frac{2}{(n-1)^{2}}+Ce^{-\frac{2}{n-1}t}+Ce^{-\frac{n-2}{n-1}t}+Ce^{-\frac{2n-4}{n-1}t}\right)w_{max}\\ \leq&\left(-\frac{2}{(n-1)^{2}}+Ce^{-\frac{2}{n-1}t}\right)w_{max}
	\end{aligned}
	\end{equation*}
which implies $w=O\Big(e^{-\frac{2}{(n-1)^{2}}t}\Big)$.
\end{proof}
\begin{rem}
Since $\upsilon=\sqrt{1+|\hn\varphi|^{2}_{\hg}}$ remains  positive and then $\ps{\partial r}{\nu}=\frac{1}{\upsilon}>0$,	Lemma \ref{bound-c1norm} implies the hypersurface solution of the inverse mean curvature flow \eqref{IMCF} remains star-shaped.
\end{rem}
\begin{lem}\label{bound-mean2}
	There exists a positive constant $C$ depending on the initial data, such that $H\geq C$.
\end{lem}
\begin{proof}
	From the evolution equation \eqref{evl-u}, Lemma \ref{C0-est} and Lemma \ref{bound-mean}, we have
	\begin{equation*}
	\begin{aligned}
	\frac{d}{dt}u_{max}&\leq-(n-1)\frac{\lambda''u_{max}}{\lambda H^{2}}\\
	&\leq\left(-\frac{1}{n-1}+Ce^{-\frac{2}{n-1}t}\right)u_{max},
	\end{aligned}
	\end{equation*}
 which implies that $\frac{\upsilon}{\lambda H}=u\leq Ce^{-\frac{1}{n-1}t}$. The assertion follows from Lemma \ref{C0-est}.
\end{proof}
\begin{rem}
The solution of the inverse mean curvature flow \eqref{IMCF} remains strictly mean convex from	Lemma \ref{bound-mean2} .
\end{rem}
\noindent With the help of Lemma \ref{bound-mean2}, we can improve the $C^{1}$-estimate in Lemma \ref{bound-c1norm}.
\begin{lem}\label{bound-c1norm2}
	We have $|\hn\varphi|_{\hg}=O(e^{-\frac{1}{n-1}t})$.
\end{lem}
\begin{proof}
From \eqref{evl-w}, we have, for $w=\frac{1}{2}|\hn\varphi|^{2}_{\hg}$,
\begin{equation}\label{evl-w1}
\frac{\partial}{\partial t}w=\frac{1}{\lambda^{2}H^{2}}\left(\hat{\sigma}^{ij} w_{ij}-\upsilon^{2}G^{l}w_{l}-2(n-1)\lambda\lambda''w-2\epsilon(n-2)w-\hat{\sigma}^{ij}\varphi_{ki}\varphi_{j}^{k}\right).
\end{equation}
Since the term $(-\epsilon)\frac{2(n-2)}{\lambda^{2}H^{2}}\leq 0$ in \eqref{evl-w1} for $\epsilon=1$ and $0$, we can infer that
\begin{equation}\label{bound-w2}
\frac{d}{dt}w_{max}\leq \left(-\frac{2}{n-1}+Ce^{-\frac{2}{n-1}t}\right)w_{max}.
\end{equation}
For $\epsilon=-1$, we have
	\begin{equation}\label{bound-w3}
	\begin{aligned}
	\frac{d}{dt}w_{max}&\leq\left(\frac{-2(n-1)\lambda''}{\lambda H^{2}}+\frac{2(n-2)}{\lambda^{2}H^{2}}\right)w_{max}\\
	&\leq\left(\frac{-2(n-1)(\kappa^{2}+m(n-2)\lambda^{-n}-q^{2}(n-2)\lambda^{2-2n})}{H^{2}}+Ce^{-\frac{2}{n-1}t}\right)w_{max}\\
	&\leq \left(-\frac{2}{n-1}+Ce^{-\frac{2}{n-1}t}\right)w_{max}.
	\end{aligned}
	\end{equation}
where the second inequlity we used Lemma \eqref{bound-mean2}. By \eqref{bound-w2}  and \eqref{bound-w3}, we have alway the estimates $|\hn\varphi|_{\hg}=O(e^{-\frac{1}{n-1}t})$ for $\epsilon=0,\pm 1$.  	
\end{proof}

\subsection{$C^2-$estimates}
Now we can give the $C^{2}$-estimates for the flow \eqref{IMCF}.
\begin{lem}\label{est-C2}
The second fundamental form $h^{j}_i$ is uniform bounded, i.e. there exists a positive constant $C$ depending on the initial condition of the flow \eqref{IMCF}, such that  $|A|\leq C$. Consequently, $|\hn^2\varphi|_{\hg}\leq C$.
\end{lem}
\begin{proof}
We define $\eta_{i}^{j}:=Hh_{i}^{j}$. By \eqref{evl-sff}, \eqref{C0-est1}, \eqref{Asy-Ricurv1} and \eqref{Asy-Ricurv2}, we have
	\begin{equation}\label{evl-sff1}
	\begin{aligned}
	\frac{\partial}{\partial t}h_{i}^{j}=&\frac{\Delta h_{i}^{j}}{H^{2}}-\frac{2\n_{i}H\n^{j}H}{H^{3}}+\frac{|A|^{2}}{H^{2}}h_{i}^{j}-\frac{2h_{i}^{k}h_{k}^{j}}{H}+(n-1)\kappa^{2}\frac{h_{i}^{j}}{H^{2}}+\frac{|A|+1}{H^{2}}O\left(e^{-\frac{n}{n-1}t}\right).
	\end{aligned}
	\end{equation}
Comibing \eqref{evl-mean} and \eqref{evl-sff1} gives
	\begin{equation}\label{evl-eta}
	\begin{aligned}
	\frac{\partial}{\partial t}\eta_{i}^{j}=&\frac{\partial}{\partial t} Hh_{i}^{j}+H\frac{\partial}{\partial t}h_{i}^{j}\\
	=&\frac{\Delta\eta_{i}^{j}}{H^{2}}-\frac{2<\n H,\n\eta_{i}^{j}>}{H^{3}}-\frac{2\n_{i}H\n^{j}H}{H^{2}}-\frac{2\eta_{i}^{k}\eta_{k}^{j}}{H^{2}}+(n-1)\kappa^{2}\frac{\eta_{i}^{j}}{H^{2}}+\frac{|\eta|+H}{H^{2}}O\left(e^{-\frac{n}{n-1}t}\right).
	\end{aligned}
	\end{equation}
Let $\mu_{max}$ be the maximal eigenvalue  of $\eta_{i}^{j}$. Since the trace of $(\eta_{i}^j)$ is positive, we have   $|\eta|\leq C\mu_{max}$. Noticing that $\n_{i}H\n^{j}H$ is non-negative definite and the mean curvature $H$ has a uniformly lower bound, we obtain
\begin{equation}
	\frac{d}{dt}\mu_{max}\leq-C\mu_{max}^{2}+C\mu_{max}+C.
\end{equation}
Therefore, we have $\mu_{max}\leq C$ for some uniform constant $C$. Since the mean curvature $H$ has a uniformly lower bound, We conclude that $|A|$ is bounded.
\end{proof}
With the above estimates, by standard parabolic Krylov and Schauder theory we can get the higher order estimate, which allows us to obtain the long time existence for the inverse mean curvarture flow \eqref{IMCF}.
\begin{lem}
	The inverse mean curvature flow \eqref{IMCF}  exists for all $t\in[0,\infty)$.
\end{lem}

\subsection{Asymptotic behaviors of the mean curvature and second fundamental form}
In this subsection, we will use the results above to improve estimates of the mean curvature and second fundamental form.
\begin{lem}\label{Asy-mean-sff}
	We have $H=(n-1)\kappa+O(te^{-\frac{2}{n-1}t})$ and $|h_{i}^{j}-\kappa\delta_{i}^{j}|\leq O(t^{2}e^{-\frac{2}{n-1}t})$.
\end{lem}
\begin{proof}
	The upper bound for the mean curvature $H$ is given  in Lemma \ref{bound-mean}. It is sufficent to bound $H$ from below, i.e.,
	\begin{equation*}
	H\geq(n-1)\kappa-O(te^{-\frac{2}{n-1}t}).
	\end{equation*}
Define $\chi=\frac{\upsilon}{H}=\lambda\varphi_{t}$. From Lemma \ref{bound-mean}, Lemma \ref{bound-c1norm} and Lemma \ref{bound-mean2},  $\chi$ is bounded from above and below. Using $\varphi_{t}=\frac{1}{G}=\frac{\chi}{\ld}$ and  \eqref{evl-var2}, we obtain
	\begin{equation}
	\begin{aligned}
	\frac{\partial}{\partial t}\chi=&\lambda\frac{\partial}{\partial t}\varphi_{t}+\lambda\lambda'\varphi_{t}^{2}\\
	=&\frac{\chi^{2}}{\upsilon^2\lambda^{2}}\left(\hs^{ij}\chi_{ij}-\frac{2}{\lambda}\hs^{ij}\lambda_{i}\chi_{j}-\upsilon^{2}G^{k}\chi_{k}\right)\\
	 &+\frac{\chi^{2}}{\lambda^{2}\upsilon^2}\left(-\frac{2\chi}{\lambda^{2}}\hs^{ij}\lambda_{i}\lambda_{j}-\frac{\chi}{\lambda}\hs^{ij}\lambda_{ij}+\frac{\upsilon^{2}\chi}{2\lambda}G^{k}\lambda_{k}\right)+\frac{\lambda'}{\lambda}\chi^{2}-\frac{n-1}{\upsilon^{2}}\frac{\lambda''}{\lambda}\chi^{3},
	\end{aligned}
	\end{equation}
	where $G^k=-\frac{2}{\upsilon^{2}}\left(G\varphi^{k}-\frac{1}{\upsilon^{2}}\varphi^{i}\varphi_{i}^{k}+\frac{1}{\upsilon^{4}}\varphi^{k}\varphi^{i}\varphi^{j}\varphi_{ij}\right)$.\\
By \eqref{C0-est1} and \eqref{bound-c1norm}, one can check that
\begin{equation*}
\begin{aligned}
\hs^{ij}\ld_i\ld_j\leq & C e^{\frac{2}{n-1}t},\\
-\hs^{ij}\ld_{ij}\leq & \lambda\lambda'(\upsilon^{2}\frac{\lambda}{\chi}-(n-1)\lambda')+Ce^{\frac{1}{n-1}t},\\
\frac{\upsilon^{2}}{2}G^{k}\ld_k\leq & Ce^{\frac{1}{n-1}t}.
\end{aligned}
\end{equation*}
	Putting these estimates together, and the boundness of $H$ and $\upsilon$,  we obtain that
	\begin{equation*}
	\begin{aligned}
	 \frac{d}{dt}\chi_{max}\leq&\frac{2\lambda'}{\lambda}\chi_{max}^{2}-\frac{n-1}{\upsilon^{2}}\frac{\lambda\lambda''+\lambda'^{2}}{\lambda^{2}}\chi_{max}^{3}+Ce^{-\frac{2}{n-1}t}\\
	\leq&2\kappa \chi_{max}^{2}-2(n-1)\kappa^{2}\chi_{max}^{3}+Ce^{-\frac{2}{n-1}t}.
	\end{aligned}
	\end{equation*}
This  implies that
	\begin{equation*}
	\frac{d}{dt}\chi_{max}\leq\frac{2}{(n-1)^{2}\kappa}-\frac{2}{n-1}\chi_{max}+Ce^{-\frac{2}{n-1}t},	
	\end{equation*}
	whenever $ \chi_{max}\geq\frac{1}{(n-1)\kappa}$.  From this fact, we deduce that $\frac{\upsilon}{H}=\chi\leq\frac{1}{(n-1)\kappa}+O(te^{-\frac{2}{n-1}t})$. Therefore, we can conclude that  $H\geq(n-1)\kappa-O(te^{-\frac{2}{n-1}t})$.

	Now we will give the proof of the second statement.
As above, we denote by $\mu_{\max}$ the maximal  eigenvalue of $(\eta_i^j)=(Hh_i^j)$. From $H=(n-1)\kappa+O(te^{-\frac{2}{n-1}t})$ and \eqref{evl-eta}, we have
	\begin{equation}
	\begin{aligned}
	\frac{d}{dt}\mu_{max}\leq&-\frac{2}{(n-1)^{2}\kappa^{2}}\mu_{max}^{2}+\frac{2}{n-1}\mu_{max}+O(te^{-\frac{n}{n-1}t})\\
	\leq&2\kappa^{2}-\frac{2}{n-1}\mu_{max}+O(te^{-\frac{2}{n-1}t}),
	\end{aligned}
	\end{equation}
	where the last inequality we used Cauchy-Schwarz inequality.\\
	Hence
	\begin{equation*}
\mu_{max}\leq(n-1)\kappa^{2}+O(t^{2}e^{-\frac{2}{n-1}t}).
	\end{equation*}
Since $H=(n-1)\kappa+O(te^{-\frac{2}{n-1}t})$ and $\eta_{i}^{j}=Hh_{i}^{j}$, we have the largest eigenvalue of $h_{i}^{j}$ is less than $\kappa+O(t^{2}e^{-\frac{2}{n-1}t})$.Furthemore, we can  infer the smallest eigenvalue is greater than $\kappa-O(t^{2}e^{-\frac{2}{n-1}t})$.
\end{proof}
\noindent From \eqref{II} and Lemma \ref{Asy-mean-sff}, we have the following Lemma.
\begin{lem}\label{est-c2}
$|\hn^2\varphi|_{\hg}\leq O\Big(t^{2}e^{-\frac{2}{n-1}t}\Big)$.
\end{lem}

\section{Proof of Theorem \ref{MainThm-Minkow} }\label{Sec:Proof of Theorem 1.1}
In this section, we will to prove the Minkowski type inequality in Theorem \ref{MainThm-Minkow}. When $\epsilon=1$, Theorem \ref{MainThm-Minkow} was proved in \cite{WangZH}. When $\epsilon=0,-1$, the proof follows from a similar argument. For the convenience of the readers, we include it here.
Now we consider a family of star-shaped hypersurface $\Sigma_t$ evolving by the inverse mean curvature flow \eqref{IMCF}. Following \cite{WangZH}, we define $\mathcal{Q}(t)$ as follows
\begin{equation}\label{QT}
\begin{aligned}
\mathcal{Q}(t)=&|\Sigma_t|^{-\frac{n-2}{n-1}}\Bigg[\int_{\Sigma_t} f H d\mu-n(n-1)\kappa^2 \int_{\Omega_t} f d v+(n-1)\vartheta_{n-1}\epsilon\left(s_0^{n-2}-\left(\frac{|\Sigma_t|}{\vartheta_{n-1}} \right)^{\frac{n-2}{n-1}}\right)\\
&+(n-1)q^2\vartheta_{n-1}\left( s_0^{2-n}-\left(\frac{|\Sigma_t|}{\vartheta_{n-1}} \right)^{-\frac{n-2}{n-1}}\right)\Bigg].
\end{aligned}
\end{equation}
\begin{lem}\label{Limit} Under the inverse mean curvature flow \eqref{IMCF}, we have
\begin{equation*}
\underset{t\rightarrow\infty}{\lim \inf}\, \mathcal{Q}(t)\geq 0.
\end{equation*}
\end{lem}
\begin{proof}

From \eqref{mean-curv}, \eqref{bound-c1norm} and Lemma \ref{est-c2}, we obtain
\begin{equation}
\begin{aligned}
H-(n-1)\kappa=&\frac{(n-1)\lambda'}{v\lambda}-\frac{1}{v\lambda}\hat{\sigma}^{ij}\varphi_{ij}-(n-1)\kappa\\
=&\frac{(n-1)\epsilon}{2\kappa}\lambda^{-2}-\frac{n-1}{2}\kappa|\hn\varphi|^{2}-\lambda^{-1}\hat{\Delta}\varphi+O(e^{-\frac{3t}{n-1}})
\end{aligned}
\end{equation}
and
\begin{equation*}
	d\mu=\ld^{n-1}\left(1+\frac12|\hn \varphi|^2_{\hg}+O\left(e^{-\frac{4}{n-1}t}\right)\right)d\mu_{\NE},
\end{equation*}	
we have
\begin{equation}\label{limit-1}
\begin{aligned}
\int_{\Sigma_t}f\left(H-(n-1)\kappa\right)d\mu=&\int_{\NE}\left(\kappa\lambda+O(e^{-\frac{1}{n-1}t})\right)
\left( \frac{(n-1)\epsilon}{2\kappa}\lambda^{-2}-\frac{n-1}{2}\kappa|\hn\varphi|^{2}-\lambda^{-1}\hat{\Delta}\varphi+O(e^{-\frac{3t}{n-1}}) \right)\\
&\cdot
\ld^{n-1}\left(1+\frac12|\hn \varphi|^2_{\hg}+O\left(e^{-\frac{4}{n-1}t}\right)\right)
d\mu_{\NE}\\
=&\frac{(n-1)\epsilon}{2}\int_{N}\lambda^{n-2}d\mu_{\NE}+\frac{n-1}{2}\int_{N}\lambda^{n-4}|\hn\lambda|^{2}_{\hg}d\mu_{\NE}+O(e^{\frac{n-3}{n-1}t}),
\end{aligned}
\end{equation}
where the integration by parts is used in the last equality.

On the other hand, from asymptotic behaviors of $f=\ld^\prime$ and $\ld^{\prime\prime}$, we have
\begin{eqnarray}
(n-1)\int_{\Sigma_{t}}\left(\kappa f-\langle\bn f,\nu\rangle\right)d\mu&\geq&\int_{\Sigma_{t}}\{(n-1)\kappa f-(n-1)|\bar{\n}f|\}d\mu\label{limit-2}\\
&=&(n-1)\int_{\Sigma_{t}}(\kappa\lambda'-\lambda'')d\mu\nonumber\\
&=&(n-1)\int_{\Sigma_{t}}\left(\kappa^{2}\lambda+\frac{\epsilon}{2}\lambda^{-1}-\kappa^{2}\lambda+O(\lambda^{-2})\right)d\mu\nonumber\\
&=&\frac{(n-1)\epsilon}{2}\int_{N}\lambda^{n-2}d\mu_{\NE}+O(e^{\frac{n-3}{n-1}t})\nonumber
\end{eqnarray}
and
\begin{equation}\label{limit-3}
\begin{aligned}
\int_{\Sigma_{t}}(n-1)\langle\bn f,\nu \rangle
d\mu -n(n-1)\kappa^{2}\int_{\Omega_{t}}fdv
=&(n-1)\left(\int_{N}\lambda^{n-1}\lambda''\mu_{\NE}-\kappa^{2}\int_{N}\int_{r_{0}}^{r}n\lambda'\lambda^{n-1}dr\mu_{\NE}\right)\\
=&(n-1)\int_{N}(\lambda^{n-1}\lambda''-\kappa^{2}\lambda^{n})d\mu_{N}+(n-1)\kappa^{2}\vartheta_{n-1} s_{0}^{n}\\
=&O(e^{-\frac{n-2}{n-1}t})+C.
\end{aligned}
\end{equation}
Therefore, \eqref{limit-1}$+$ \eqref{limit-2}$+$\eqref{limit-3} implies
\begin{eqnarray}
\int_{\Sigma_{t}}fHd\mu-n(n-1)\kappa^{2}\int_{\Omega_{t}}fd\nu&\geq&(n-1)\epsilon\int_{N}\lambda^{n-2}d\mu_{N}+\frac{n-1}{2}\int_{N}\lambda^{n-4}|\hn\lambda|^{2}d\mu_{N}+O(e^{\frac{n-3}{n-1}t})\nonumber\\
&\geq&(n-1)\epsilon\vartheta_{n-1}^{\frac{1}{n-1}}\left(\int_{N}\lambda^{n-1}d\mu_{N}\right)^{\frac{n-2}{n-1}}+O(e^{\frac{n-3}{n-1}t}),
\end{eqnarray}
where the last inequality above is from a non-sharp version of Beckner-Sobolev type inequality in \cite{Beckner} when $\epsilon=1$ (see Proposition 21 in \cite{WangZH});  the inequality is trivial when $\epsilon=0$; similarly in \cite{GWWX}, when $\epsilon=-1$, we use the H\"{o}lder inequality
$$
\int_{N} \ld^{n-2}\leq\vartheta_{n-1}^{\frac 1{n-1}}\left( \int_{N} \ld^{n-1}\right)^{\frac{n-2}{n-1}}.
$$
Noting that
\begin{equation*}
|\Sigma_{t}|=\int_{N}\lambda^{n-1}(1+O(e^{-\frac{2t}{n-1}}))d\mu_{N}=\int_{N}\lambda^{n-1}d\mu_{N}+O(e^{\frac{n-3}{n-1}t}),
\end{equation*}
we get
\begin{equation*}
{\lim_{t\rightarrow\infty}\inf}\frac{\int_{\Sigma_t}fHd\mu-n(n-1)\kappa^{2}\int_{\Omega_t}fd\mu}{|\Sigma_t|^{\frac{n-2}{n-1}}}\geq(n-1)\epsilon\vartheta_{n-1}^{\frac{1}{n-1}}
\end{equation*}
and
\begin{equation*}
(n-1)\lim_{t\rightarrow\infty}{|\Sigma_t|^{-\frac{n-2}{n-1}}}\left(\vartheta_{n-1}\epsilon \left(s_0^{n-2}-\left(\frac{|\Sigma_t|}{\vartheta_{n-1}} \right)^{\frac{n-2}{n-1}}\right)+q^{2}\vartheta_{n-1}\left(s_0^{2-n}-\left(\frac{|\Sigma_t|}{\vartheta_{n-1}} \right)^{\frac{2-n}{n-1}}\right)\right)=-(n-1)\epsilon\vartheta_{n-1}^{\frac{1}{n-1}},
\end{equation*}
which completes the proof of Lemma \ref{Limit}.
\end{proof}

Now we prove that $\mathcal{Q}(t)$ has the following monotonicity along the inverse mean curvature flow \eqref{IMCF}.
\begin{lem}\label{Mono}
The quantity $\mathcal{Q}(t)$ is monotone nonincreasing along the flow \eqref{IMCF}.
\end{lem}
\begin{proof}
Using the evolution equations \eqref{evl-measure} and \eqref{evl-mean},  the identity $ \Delta f=\bar{\Delta}f-\bar{\n}^2f(\nu,\nu)-H\langle \bar{\n}f,\nu\rangle$ and the  sub-static equation \eqref{Static-Equ}, we obtain
\begin{equation}\label{fH}
\begin{aligned}
\frac{d}{dt} \int_{\Sigma_t} f H d\mu=&-\int_{\Sigma_t}\frac 1H\Delta fd\mu-\int_{\Sigma_t}\frac{f}{H}\left(|A|^2+\RicN(\nu,\nu)\right)d\mu+\int_{\Sigma_t}\left(\langle \bar{\n}f,\nu\rangle+fH\right)d\mu\\
&+\int_{\Sigma_t}\left(\langle \bar{\n}f,\nu\rangle+fH\right)d\mu\\
=&-\int_{\Sigma_t}\frac 1H\left(\bar{\Delta}f-\bar{\n}^2f(\nu,\nu)+f\RicN(\nu,\nu))   \right)d\mu-\int_{\Sigma_t}\frac{f}{H}|A|^2d\mu\\
&+\int_{\Sigma_t}\left(2\langle \bar{\n}f,\nu\rangle+fH\right)d\mu\\
\leq&\int_{\Sigma_t}\left(2\langle \bar{\n}f,\nu\rangle+\frac{n-2}{n-1}fH\right)d\mu.
\end{aligned}
\end{equation}

The hypersurce $\Sigma_t$ can be seen as the graph  $\Sigma_t=\{(\omega,s_t(\omega))|\omega\in \NE\}$ of the space form $\NE$. By the divergece theorem, we obtain
\begin{equation}\label{bnf}
\int_{\Sigma_t}\langle \bar{\n}f,\nu\rangle d\mu=\kappa^2\int_{\NE} s_t^nd\mu_{\NE}+m(n-2)\vartheta_{n-1}-(n-2)q^2\int_{\NE} s_t^{2-n}d\mu_{\NE}.
\end{equation}
Since
\begin{equation*}\label{}
n\int_{\Omega_t} fd v=n\int_{\NE} \int_{s_0}^{s_t(\omega)}f\cdot \frac{s^{n-1}}{f}ds d\mu_{\NE}=\int_{\NE} s_t^n d\mu_{\NE}-s_0^n\vartheta_{n-1},
\end{equation*}
we have
\begin{equation}\label{}
\begin{aligned}
\int_{\Sigma_t}\langle \bar{\n}f,\nu\rangle d\mu=&n\kappa^2\int_{\Omega_t}fd v+\kappa^2 s_0^n\vartheta_{n-1}+m(n-2)\vartheta_{n-1}-(n-2)q^2\int_{\NE} s_t^{2-n}d\mu_{\NE}.
\end{aligned}
\end{equation}
Therefore,
\begin{equation}\label{}
\begin{aligned}
\frac{d}{dt} \int_{\Sigma_t} f H d\mu\leq&
2n\kappa^2\int_{\Omega_t}fd v+2\kappa^2 s_0^n\vartheta_{n-1}+2m(n-2)\vartheta_{n-1}-2(n-2)q^2\int_{\NE} s_t^{2-n}d\mu_{\NE}\\
&+\frac{n-2}{n-1}\int_{\Sigma_t}fH d\mu.
\end{aligned}
\end{equation}
From a  Heintze-Karcher type
inequality proved by Brendle \cite{Brendle}, we have
\begin{equation}\label{HK}
\frac{d}{dt} \int_{\Omega_t} f d v= \int_{\Sigma_t} \frac{f }{H} d\mu\geq \frac{n}{n-1}\int_{\Omega_t} f d v+\frac{1}{n-1}s_0^n\vartheta_{n-1}.
\end{equation}
Hence,
\begin{equation}\label{fh1}
\begin{aligned}
\frac{d}{dt}& \left(\int_{\Sigma_t} f H d\mu-n(n-1)\kappa^2 \int_{\Omega_t} f d v  \right)\\
\leq&  \frac{n-2}{n-1}\left(\int_{\Sigma_t} f H d\mu-n(n-1)\kappa^2 \int_{\Omega_t} f d v  \right)\\
&+2m(n-2)\vartheta_{n-1}-2(n-2)q^2\int_{\NE} s_t^{2-n}d\mu_{\NE}\\
=&\frac{n-2}{n-1}\left(\int_{\Sigma_t} f H d\mu-n(n-1)\kappa^2 \int_{\Omega_t} f d v  \right.\\
&\left.+ (n-1)\vartheta_{n-1}\left(\epsilon s_0^{n-2}+q^2s_0^{2-n}\right)-2(n-1)q^2\int_{\NE} s_t^{2-n}d\mu_{\NE}\right).
\end{aligned}
\end{equation}
By the H?lder inequality, we have
\begin{equation*}
\vartheta_{n-1}^{\frac{2n-3}{n-1}} \left(\int_{\NE} s_t^{n-1}d\mu_{\NE} \right)^{-\frac{n-2}{n-1}}\leq \int_{\NE} s_t^{2-n}d\mu_{\NE}.
\end{equation*}
From $ |\Sigma_t|\geq \int_{\NE} s_t^{n-1}d\mu_{\NE} $, we get
\begin{equation}\label{fh2}
\vartheta_{n-1} \left(\frac{|\Sigma_t|}{\vartheta_{n-1}} \right)^{-\frac{n-2}{n-1}}\leq \int_{\NE} s_t^{-(n-2)}d\mu_{\NE}.
\end{equation}
From \eqref{evl-measure}, \eqref{fh1} and \eqref{fh2},
we have
\begin{equation}\label{fh3}
\begin{aligned}
&\frac{d}{dt} \left(\int_{\Sigma_t} f H d\mu-n(n-1)\kappa^2 \int_{\Omega_t} f d v+(n-1)\epsilon\vartheta_{n-1}\left(s_0^{n-2}-\left(\frac{|\Sigma_t|}{\vartheta_{n-1}} \right)^{\frac{n-2}{n-1}}\right)\right.\\
&\left.+(n-1)q^2\vartheta_{n-1}\left( s_0^{2-n}-\left(\frac{|\Sigma_t|}{\vartheta_{n-1}} \right)^{-\frac{n-2}{n-1}}\right)\right)\\
\leq &\frac{n-2}{n-1} \left(\int_{\Sigma_t} f H d\mu-n(n-1)\kappa^2 \int_{\Omega_t} f d v+(n-1)\epsilon\vartheta_{n-1}\left(s_0^{n-2}-\left(\frac{|\Sigma_t|}{\vartheta_{n-1}} \right)^{\frac{n-2}{n-1}}\right)\right.\\
&\left.+(n-1)q^2\vartheta_{n-1}\left( s_0^{2-n}-\left(\frac{|\Sigma_t|}{\vartheta_{n-1}} \right)^{-\frac{n-2}{n-1}}\right)\right),
\end{aligned}
\end{equation}
which completes the proof of Lemma \ref{Mono}.
\end{proof}
\begin{proof}[Proof of Theorem \ref{MainThm-Minkow}]
By use of Lemma \ref{Limit} and Lemma \ref{Mono}, we have
\begin{equation*}
Q(0)\geq\lim_{t\rightarrow\infty}\inf Q(t)\geq0,
\end{equation*}
which gives the proof of Theorem \ref{MainThm-Minkow}.
\end{proof}
\section{Proof of Theorem \ref{AF-ineq}} \label{Sec:Proof of Theorem 1.3}
In this section, following the arguments in \cite{dLG3,GWWX},  we will use Theorem \ref{MainThm-Minkow} to give the proof of  Theorem \ref{AF-ineq}.

We define
\begin{equation}\label{Jfuncitonal}
\mathcal{J}(\Sigma_t)=n\kappa^2\int_{\Omega_t}  f dv,
\end{equation}
\begin{equation}\label{Kfunctional}
\mathcal{K}(\Sigma_t)=\kappa^2\vartheta_{n-1}
\left(\left(\frac{|\Sigma_t|}{\vartheta_{n-1}} \right)^{\frac{n}{n-1}}-\left(\frac{|\partial P|}{\vartheta_{n-1}} \right)^{\frac{n}{n-1}}\right),
\end{equation}
\begin{equation}\label{Lfunctional}
\begin{aligned}
\mathcal{L}(\Sigma_t)=&|\Sigma_t|^{-\frac{n-2}{n-1}}\left(\int_{\Sigma_t} f H d\mu-(n-1)\mathcal{K}(\Sigma_t) +(n-1)\epsilon\vartheta_{n-1} s_0^{n-2}\right.\\
&\left.+(n-1)q^2\vartheta_{n-1}\left( s_0^{2-n}-\left(\frac{|\Sigma_t|}{\vartheta_{n-1}} \right)^{-\frac{n-2}{n-1}}\right)\right).
\end{aligned}
\end{equation}
By using Lemma \ref{evl-equations} and a Heintze-Karcher inequality \eqref{HK}, we can check that
\begin{equation*}
\frac{d}{dt}\left(\frac{\mathcal{J}(\Sigma_t)-\mathcal{K}(\Sigma_t) }{|\Sigma_t|^{\frac{n-2}{n-1}}} \right) \geq 0.
\end{equation*}
By using Lemma \ref{evl-equations}, \eqref{fH}, \eqref{bnf}, \eqref{fh3} and $\epsilon+\kappa^{2}s_{0}^{2}-2ms_{0}^{2-n}+q^{2}s_{0}^{4-2n}=0$, we have
\begin{equation*}
\begin{aligned}
\frac{d}{dt}&\left(\int_{\Sigma_t} f H d\mu-(n-1)\mathcal{K}(\Sigma_t) +(n-1)\vartheta_{n-1}\epsilon s_0^{n-2}\right.\\
&\left.+(n-1)q^2\vartheta_{n-1}\left( s_0^{2-n}-\left(\frac{|\Sigma_t|}{\vartheta_{n-1}} \right)^{-\frac{n-2}{n-1}}\right)\right)\\
\leq & 2\mathcal{J}(\Sigma_t)-n\mathcal{K}(\Sigma_t)-(n-2)\kappa^2 s_0^n\vartheta_{n-1}+2m(n-2)\vartheta_{n-1}-2(n-2)q^2\int_{\NE} s_t^{2-n}d\mu_{\NE}\\
&+\frac{n-2}{n-1}\int_{\Sigma_t}f H d\mu+(n-2)q^2\vartheta_{n-1}\left(\frac{|\Sigma_t|}{\vartheta_{n-1}} \right)^{-\frac{n-2}{n-1}}\\
\leq & 2\left(\mathcal{J}(\Sigma_t)-\mathcal{K}(\Sigma_t)\right)+\frac{n-2}{n-1}\left(\int_{\Sigma_t} f H d\mu-(n-1)\mathcal{K}(\Sigma_t) +(n-1)\vartheta_{n-1}\epsilon s_0^{n-2}\right.\\
&\left.+(n-1)q^2\vartheta_{n-1}\left( s_0^{2-n}-\left(\frac{|\Sigma_t|}{\vartheta_{n-1}} \right)^{-\frac{n-2}{n-1}}\right)\right)\\
&-(n-2)\vartheta_{n-1}\epsilon s_0^{n-2}-(n-2)q^2\vartheta_{n-1} s_0^{2-n}-(n-2)\kappa^2 s_0^n\vartheta_{n-1}+2m(n-2)\vartheta_{n-1}\\
=&2\left(\mathcal{J}(\Sigma_t)-\mathcal{K}(\Sigma_t)\right)+\frac{n-2}{n-1}|\Sigma_{t}|^{\frac{n-2}{n-1}}\mathcal{L}(\Sigma_{t}),
\end{aligned}
\end{equation*}
which  means that
\begin{equation}\label{Mono-L}
\frac{d}{dt}\mathcal{L}(\Sigma_{t})\leq 2\frac{\mathcal{J}(\Sigma_t)-\mathcal{K}(\Sigma_t) }{|\Sigma_t|^{\frac{n-2}{n-1}}}
\end{equation}

It is  sufficent to prove Theorem \ref{AF-ineq} when the initial surface $\Sigma_0$ satisfies $\mathcal{J}(\Sigma_{0})<\mathcal{K}(\Sigma_{0}),$
otherwise the assertion follows directly from Theorem \ref{MainThm-Minkow}.
In order to prove Theorem \ref{AF-ineq}, we  divide the proof into two cases.

\noindent\textbf{Case 1:}
There exists some $t_0\in(0,\infty)$ such that $\mathcal{J}(\Sigma_{t_0})=\mathcal{K}(\Sigma_{t_0})$ and $\mathcal{J}(\Sigma_{t})-\mathcal{K}(\Sigma_{t})\leq 0$ for $t\in [0,t_0] $.
From \eqref{Mono-L} and \eqref{Minkow0}, we obtain
\begin{equation}\label{}
\begin{aligned}
\mathcal{L}(\Sigma_{0})\geq& \mathcal{L}(\Sigma_{t_0})\\
=&|\Sigma_{t_0}|^{-\frac{n-2}{n-1}}\left(\int_{\Sigma_{t_0}} f H d\mu-(n-1)\mathcal{J}(\Sigma_{t_0}) +(n-1)\epsilon\vartheta_{n-1} s_0^{n-2}\right.\\
&\left.+(n-1)q^2\vartheta_{n-1}\left( s_0^{2-n}-\left(\frac{|\Sigma_{t_0}|}{\vartheta_{n-1}} \right)^{-\frac{n-2}{n-1}}\right)\right)\\
\geq & (n-1)\epsilon \vartheta_{n-1}^{\frac{1}{n-1}}.
\end{aligned}
\end{equation}

\noindent\textbf{Case 2:}
For all $t\in[0,\infty)$, we have
$$
\mathcal{J}(\Sigma_{t})-\mathcal{K}(\Sigma_{t})< 0.
$$
Since $\mathcal{L}(t)$ is monotone non-increasing in $t\in[0,\infty)$ from \eqref{Mono-L}, we obtain
\begin{equation}\label{Mono-0}
\mathcal{L}(\Sigma_{0})\geq\mathcal{L}(\Sigma_{\infty})=\lim_{t\rightarrow\infty}\frac{\int_{\Sigma_{t}}fHd\mu-(n-1)\mathcal{K}(\Sigma_{t})}{|\Sigma_{t}|^{\frac{n-2}{n-1}}}.
\end{equation}
By \eqref{limit-1}, we have
\begin{equation}\label{Mono-1}
	\begin{aligned}
\int_{\Sigma_{t}}fHd\mu-(n-1)\mathcal{K}(\Sigma_{t})=&\int_{\Sigma_{t}}f(H-(n-1)\kappa)d\mu+(n-1)\left(\int_{\Sigma_{t}}\kappa fd\mu-\mathcal{K}(\Sigma_{t})\right)\\
=&\frac{(n-1)\epsilon}{2}\int_{\NE}\lambda^{n-2}d\mu_{\NE}+\frac{n-1}{2}\int_{\NE}\lambda^{n-4}|\hn\lambda|^{2}d\mu_{\NE}\\
&+(n-1)\left(\int_{\Sigma_{t}}\kappa fd\mu-\mathcal{K}(\Sigma_{t})\right)
+O\left(e^{\frac{n-3}{n-1}t}\right).
	\end{aligned}
\end{equation}
By the H?lder inequality, we have
\begin{equation}\label{Holder1}
\left(\frac{|\Sigma_{t}|}{\vartheta_{n-1}}\right)^{\frac{n}{n-1}}\leq\frac{\int_{\NE}\lambda^{n}\upsilon^{\frac{n}{n-1}}d\mu_{\NE}}{\vartheta_{n-1}}.
\end{equation}
By \eqref{C0-est}, Lemma \ref{bound-c1norm2} and \eqref{Holder1}, we obtain
\begin{equation}\label{Mono-2}
\begin{aligned}
(n-1)\left(\int_{\NE}\kappa fd\mu-\mathcal{K}(\Sigma_{t})\right)=&(n-1)\left(\int_{\Sigma_{t}}\kappa f\lambda^{n-1}\upsilon d\mu_{\NE}-\kappa^2\vartheta_{n-1}\left(\frac{|\Sigma_t|}{\vartheta_{n-1}} \right)^{\frac{n}{n-1}}  \right)\\
\geq&(n-1)\int_{\NE}\left(\kappa f \lambda^{n-1}\upsilon-\kappa^{2}\lambda^{n}\upsilon^{\frac{n}{n-1}}\right)d\mu_{\NE}) \\
=&(n-1)\int_{\NE}\left(\frac{\epsilon}{2}\lambda^{n-2}-\frac{1}{2(n-1)}\lambda^{n-4}|\hn\lambda|^{2}\right)d\mu_{\NE}+O(e^{\frac{n-3}{n-1}t}).
\end{aligned}
\end{equation}
This implies
\begin{equation}\label{Mono-3}
\begin{aligned}
\int_{\Sigma_{t}}fHd\mu-(n-1)\mathcal{K}(\Sigma_{t})\geq& (n-1)\epsilon\int_{N}\lambda^{n-2}d\mu_{\NE}+\frac{n-2}{2}\int_{N}\lambda^{n-4}|\hn\lambda|^{2}d\mu_{\NE}+O\left(e^{\frac{n-3}{n-1}t}\right)\\
\geq &(n-1)\epsilon\vartheta_{n-1}^{\frac{1}{n-1}}\left(\int_{N}\lambda^{n-1}d\mu_{N}\right)^{\frac{n-2}{n-1}}+O(e^{\frac{n-3}{n-1}t}),
\end{aligned}
\end{equation}
where the Beckner-Sobolev type inequality \cite{Beckner,BHW} is used in second inequality  when $\epsilon=1$; the inequality is trivial when $\epsilon=0$; it comes from the H\"{o}lder inequality when $\epsilon=-1$.

Finally, by \eqref{Mono-0}, \eqref{Mono-1}, \eqref{Mono-2} and \eqref{Mono-3}, we deduce that
\begin{equation*}
\begin{aligned}
\mathcal{L}(\Sigma_{0})
\geq&(n-1)\vartheta_{n-1}^{\frac{1}{n-1}}\epsilon\lim_{t\rightarrow\infty}\frac{\left(\int_{\NE}\lambda^{n-1}d\mu_{\NE}\right)^{\frac{n-2}{n-1}}}{|\Sigma_{t}|^{\frac{n-2}{n-1}}}\\
=&(n-1)\epsilon\vartheta_{n-1}^{\frac{1}{n-1}},
\end{aligned}
\end{equation*}
which completes the proof of Theorem \ref{AF-ineq}.

\section{Proofs of Theorem 1.5 and Theorem 1.6}\label{Sec:Proofs of Theorem 1.5 and 1.6}

We will compute the mass of the ALH graph over  Reissner-Nordstr\"om-AdS manifold. With the Alexandrov-Fenchel inequality in Theorem \ref{AF-ineq}, we now can  give a Penrose inequality  of ALH mass.

\begin{lem}
The   Reissner-Nordstr\"om-AdS manifold $(P,\barg)$ is a ALH manifold with ALH mass
\begin{equation}\label{mass-1}
m(P,\barg)=m.
\end{equation}
\end{lem}
\begin{proof}
Let $\{x^{a}\}_{a=1}^{n-1}$ be a local coordinate system on $\NE$ and $x^{0}=s$. According to \eqref{ALH-metric}, one can check that the Christoffel symbols with respect to the metric $b_{\epsilon}$ are
\begin{equation}
\begin{cases}
&\Gamma_{00}^{0}=-\frac{\kappa^{2}s}{\kappa^{2}s^{2}+\epsilon},\\
&\Gamma_{ab}^{0}=-sV^2_{\epsilon}\hg_{ab},\qquad a,b>0,\\
&\Gamma_{a0}^{b}=s^{-1}\delta_{a}^{b},\qquad a,b>0,\\
&\Gamma_{a0}^{0}=\Gamma_{00}^{a}=0,\qquad a>0.
\end{cases}
\end{equation}
Since $d\mu=s^{n-1}d\mu_{\NE}, e=(\frac{ds^{2}}{f^{2}}-\frac{ds^{2}}{V_{\epsilon}^{2}})$, one can compute directly
\begin{equation}
\begin{aligned}
m(M,g)=&\frac{1}{2(n-1)\vartheta_{n-1}}\lim_{s\rightarrow\infty}\int_{N_{s}}\left(V_{\epsilon}(div^{b_{\epsilon}}e-dtr^{b_{\epsilon}}e)+(tr^{b_{\epsilon}}e)dV_{\epsilon}-e(\n^{b_{\epsilon}}V_{\epsilon},\cdot)\right)\nu d\mu_{b_{\epsilon}}\\
=&\frac{1}{2(n-1)\vartheta_{n-1}}\int_{N_{\epsilon}}\left(2m(n-1)+\frac{2m}{\kappa}-\frac{2m}{\kappa}\right)d\mu_{\NE}\\
=&m.
\end{aligned}
\end{equation}
and
\begin{equation}\label{ALHG}
||(\Phi^{-1})*\barg-b_{\epsilon}||_{b_{\epsilon}}+||\n^{b_{\epsilon}}((\Phi^{-1})*\barg)||_{b_{\epsilon}}=O(s^{-n})
\end{equation}
By using the static equation \eqref{Static-Equ}, we have
\begin{equation*}
R_{g}=-n(n-1)\kappa^{2}+q^{2}(n-2)(n-1)s^{2-2n}
\end{equation*}
which implies
\begin{equation*}
\begin{aligned}
\int_{P}V_{\epsilon}|R_{g}+ n(n-1)\kappa^{2}|d\mu_{M}
=q^{2}(n-2)(n-1)\int_{s_0}^{\infty}\int_{\NE}V_{\epsilon}s^{2-2n}V_{\epsilon}s^{n-1}dsd\mu_{\NE}
<\infty.
\end{aligned}
\end{equation*}
Hence $(P,\barg)$ is a ALH manifold.
\end{proof}
\begin{lem}
We have
\begin{equation}
m(M,g)=m+\frac{1}{2(n-1)\vartheta_{n-1}}\lim_{s\rightarrow\infty}\int_{N_{s}}\left(f(div^{\barg}\hat{e}-{\rm d}\ tr^{\barg}\hat{e})+(tr^{\barg}\hat{e}){\rm d}f-\hat{e}(\bn f,\cdot)\right)\hat{\nu}d\mu
\end{equation}
where $\hat{e}=(\Phi^{-1})*g-\barg$, $\hat{\nu}$ is the outward unit normal of $N_{s}$ induced by $\barg$ and $d\mu$ is the area element induced by $\barg$.
\end{lem}
\begin{proof}
Since $e=\hat{e}+\barg-b_{\epsilon}$,
we have
$$
\begin{aligned}
m(M,g)&=m(P,\barg)+\frac{1}{2(n-1)\vartheta_{n-1}}\lim_{s\rightarrow\infty}\int_{N_{s}}\left(V_{\epsilon}(div^{b_{\epsilon}}\hat{e}-dtr^{b_{\epsilon}}\hat{e})+(tr^{b_{\epsilon}}\hat{e})dV_{\epsilon}-\hat{e}(\n^{b_{\epsilon}}V_{\epsilon},\cdot)\right)\nu d\mu_{b_{\epsilon}}\\
&=m+\frac{1}{2(n-1)\vartheta_{n-1}}\lim_{s\rightarrow\infty}\int_{N_{s}}\left(V_{\epsilon}(div^{b_{\epsilon}}\hat{e}-dtr^{b_{\epsilon}}\hat{e})+(tr^{b_{\epsilon}}\hat{e})dV_{\epsilon}-\hat{e}(\n^{b_{\epsilon}}V_{\epsilon},\cdot)\right)\nu d\mu_{b_{\epsilon}}.
\end{aligned}
$$
Recall that the  Reissner-Nordstr\"om-AdS manifold is ALH, we have \eqref{ALHG}. 
Then one can replace $V_{\epsilon},b_{\epsilon},\nu, d\mu_{b_{\epsilon}}$ by,  $f,\barg,\hat{\nu}, d\mu$, respectively, will not change the value of the limit, i.e.
$$
m(M,g)=m+\frac{1}{2(n-1)\vartheta_{n-1}}\lim_{s\rightarrow\infty}\int_{N_{s}}\left(f(div^{\barg}\hat{e}-dtr^{\barg}\hat{e})+(tr^{\barg}\hat{e})df-\hat{e}(\bn f,\cdot)\right)\hat{\nu} d\mu.
$$
\end{proof}
\begin{proof}[\textbf{Proof of Theorem \ref{Penrose0}}] The proof of this theorem follows in the spirit of the one in \cite{dLG2,GWWX}.
	Denote the outward unit normal of $(M,g)\subset (Q,\tg)$ by $\xi$ and let $B=-\bn\xi$ be the second fundamental form of $M$. The Newton tensor is inductively given  by
	$$
	T_{r}=S_{r}I-BT_{r-1},\ T_{0}=I,
	$$
	where $S_{r}$ is the $r^{th}$ elementary symmetric polynomial of principle curvature  of $M$ with respect to the unit normal $\xi$.
	
	Let $\{x_{0},...,x_{n-1}\}$  be a local coordinate system of $P$. The induced metric $g_{ij}$ and the unit normal vector field $\xi$ of $M$, respectively, are given by
	\begin{align}
	g_{ij}=&f^{2}u_{i}u_{j}+\barg_{ij},\label{metric}\\
	\xi=&\frac{1}{\sqrt{1+f^{2}|\bn u|^{2}}}\left(f^{-1}\frac{\partial }{\partial t}-f\bn u\right),\label{normal-M}
	\end{align}
	where $\bn$ is the covariant derivative with respect to $\barg$, $u_k=\bn_k u, u^{i}=\barg^{ik}u_k$ and $|\bn u|^{2}=u_{i}u_{j}\barg^{ij}$.
	The  components of second foundamental form $B_{ij}$ is given by
	\begin{equation}\label{sff-M}
	B_{ij}=\frac{f}{\sqrt{1+f^{2}|\bn u|^{2}}}\left(\bn_{i}\bn_{j}u+\frac{u_{i}f_{j}+u_{j}f_{i}}{f}+f\left\langle\bn u,\bn f\right\rangle_{\barg}u_{i}u_{j}\right).
	\end{equation}
	From \eqref{normal-M}, we deduce that the tangential part of $\frac{\partial}{\partial t}$ is
	\begin{equation}
	\begin{aligned}
	\left(\frac{\partial}{\partial t}\right)^{T}=&\frac{f^{2}\bn u}{1+f^{2}|\bn u|^{2}}+\frac{f^{2}|\bn u|^{2}}{1+f^{2}|\bn u|^{2}}\partial_t\\
	=&\frac{f^{2}\bn^{i}u}{1+f^{2}|\bn u|^{2}}\left(\bn_iu\partial_t+\partial_i\right)
	\end{aligned}
	\end{equation}
	From \eqref{metric}and  \eqref{sff-M}, we get
	\begin{equation}\label{jifen}
	\begin{aligned}
	&\left[f(div^{\barg}\hat{e}-{\rm d}\ tr^{\barg}\hat{e})+(tr^{\barg}\hat{e}){\rm d}f-\hat{e}(\bn f,\cdot)\right]_{i}\\
	&=f^{2}\left\langle \bn u,\bn f\right\rangle_{\barg}u_{i}-f^{3}\bn^{k}\bn_{i}u\bn_{k}u+f^{3}\bar{\Delta} u\bn_{i}u-f^{2}|\bn u|^{2}f_{i}\\
	&=(B_{ml}\barg^{ml}\barg_{ik}-B_{ik})\barg^{kj}u_{j}\cdot f^{2}\sqrt{1+f^{2}|\bn u|^{2}}
	\end{aligned}
	\end{equation}
	Note that since $M$ is ALH, by Definition \ref{ALH graph metric}, we have for some $\tau>\frac{n}{2}$
	\begin{equation*}
	f^{2}|\bn u|^{2}=O(s^{-2\tau}),\ g_{ij}=\barg_{ij}+O(s^{-2\tau}),\ g^{ij}=\barg^{ij}+O(s^{-2\tau-2})
	\end{equation*}
	Since $d\mu=s^{n-1}d\mu_{N_{\epsilon}}$, we have
	\begin{equation*}
	d\mu_{N_{s}}=\sqrt{det\left(s^{2}\hg_{ab}+f^{2}u_{a}u_{b}\right)}dx^{1}\cdots dx^{n-1}=d\mu(1+O(s^{-2\tau-2}))
	\end{equation*}
	Because the tangent space of $\text{graph}|_{N_{s}}u$ is spanned by $e_{a}:=u_{a}\partial_t+\partial_{a},a=1,\cdots,n-1$. The normal vector of  $\text{graph}|_{N_{s}}u$ can be set as $e_{0}^{\perp}=e_{0}-t_{a}e_{a}$, where $e_{0}=u_{s}\partial_t +\partial_{s}$. Then from $g(e_{a},e_{0}^{\perp})=0$, we have $t_{a}=O(s^{-2\tau-2})$. Note that $|e_{a}|=O(s)$, $|e_{0}|=O(s^{-1})$, $f\partial_{s}=\hat{\nu}$, so we can see  $$\nu_{N_{s}}=\frac{e_{0}^{\perp}}{|e_{0}^{\perp}|}=\frac{O(s^{-\tau-2})\partial_t+\partial_{s}+O(s^{-2\tau-2})\Sigma\partial_{a}}{\sqrt{O(s^{-2\tau-2})+\frac{1}{f^{2}}+O(s^{-4\tau-2})}}=\hat{\nu}+O(s^{-2\tau}).$$
	Hence from \eqref{jifen} and above,
	when $s\rightarrow\infty$,
	\begin{equation}
	\begin{aligned}
	&\lim_{s\rightarrow\infty}\int_{N_{s}}\left(f(div^{\barg}\hat{e}-{\rm d}\ tr^{\barg}\hat{e})+(tr^{\barg}\hat{e}){\rm d}f-\hat{e}(\bn f,\cdot)\right)\hat{\nu} d\mu\\
	=&\lim_{s\rightarrow\infty}\int_{N_{s}}(B_{ml}g^{ml}\delta_{i}^{j}-B_{i}^{j})\frac{f^{2}\bn^{i}u}{1+f^{2}|\bn u|^{2}}\left(\nu_{N_{s}}\right)_jd\mu_{g}\\
	=&\lim_{s\rightarrow\infty}\int_{N_{s}}\left\langle T_{1}\left(\frac{\partial}{\partial t}\right)^{T},\nu_{N_{s}}\right\rangle_gd\mu_{g}.
	\end{aligned}
	\end{equation}
	By integration by parts , we get (see \cite{GWW4})
	\begin{equation}\label{integrate by parts}
	\lim_{s\rightarrow\infty}\int_{N_{s}}\left<T_{1}\left(\frac{\partial}{\partial t}\right)^{T},\nu_{N_{s}}\right>_{g}d\mu_{g}=\int_{M}div_{g}\left(T_{1}\left(\frac{\partial}{\partial t}\right)^{T}\right)d\mu_{M}+\int_{\Sigma}\left<T_{1}\left(\frac{\partial}{\partial t}\right)^{T},\nu_{\Sigma}\right>_{g}d\mu_{\Sigma}
	\end{equation}
	where $\nu_{\Sigma}, d\mu_{\Sigma}$ are the unit outer normal and area element of $\text{graph}|_{\Sigma}u$ induced by metric $g$, respectively.
	
	Using the fact that $\frac{\partial}{\partial t}$ is a Killing vector field, one can check (refer \cite{ALM} for the proof)
	\begin{equation}\label{div1}
	div_{g}\left(T_{1}\left(\frac{\partial}{\partial t}\right)^{T}\right)=\left\langle div_{g}T_{1},\left(\frac{\partial}{\partial t}\right)^{T}\right\rangle_{g}+2S_{2}\left\langle\frac{\partial}{\partial t},\xi\right\rangle.
	\end{equation}
	A direct computation (see \cite{ALM}) gives
	\begin{equation}
	div_{g}(T_{r})=-Bdiv_{g}(T_{r-1})-\sum_{i}\left(R_{\tg}(\xi,T_{r-1}(\partial_{i}))\partial_i\right)^{T}
	\end{equation}
	where $\{\partial_{i}\}$ is a tangent frame of $M$, $R_{\tg}(\cdot,\cdot)$ is the Riemannian curvature tensor of $(Q,\tg)$. In particular, for  $r=1$ we have
	\begin{equation}\label{div2}
	\left\langle div_{g}T_{1},\left(\frac{\partial}{\partial t}\right)^{T}\right\rangle_{g}=\Ric_{\tg}(\xi,\left(\frac{\partial}{\partial t}\right)^{T}).
	\end{equation}
From the Gauss equation, we have
\begin{equation}
2S_{2}=R_{g}-R_{\tg}+2Ric_{\tg}(\xi,\xi)
\end{equation}
where $	R_{\tg}=-n(n+1)\kappa^{2}-(n-2)(n-3)q^{2}\lambda^{2-2n}$ is the scalar curvature of the manifold $(Q,\tg)$.
After a change of variable as in \eqref{metric-RN}, the metric $\tg$ has  the following form $$\tg=f^{2}dt^{2}+dr^{2}+\lambda^{2}\hg,\ f=\lambda'.$$
Now we have
\begin{equation}\label{Ricci1}
Ric_{\tg}(\xi,\xi)=-nk^{2}+\frac{(n-2)q^{2}\lambda^{2-2n}}{1+f^{2}|\bn u|^{2}}\left(-(n-2)(1+f^{2}u_{r}^{2})+f^{2}(|\bn u|^{2}-u_{r}^{2})\right),
\end{equation}
and
\begin{equation}\label{Ricci2}
Ric_{\tg}(\xi,\left(\frac{\partial}{\partial t}\right)^{T})=-\left\langle\frac{\partial}{\partial t},\xi\right\rangle \left(Ric_{\tg}(\xi,\xi)+n\kappa^{2}+(n-2)^{2}q^{2}\lambda^{2-2n}\right).
\end{equation}
Combining   \eqref{div1}, \eqref{div2}, \eqref{Ricci1} and \eqref{Ricci2},  we obtain
\begin{equation}\label{div}
\begin{aligned}
div_{g}\left(T_{1}\left(\frac{\partial}{\partial t}\right)^{T}\right)
=&2S_{2}\left\langle\frac{\partial}{\partial t},\xi\right\rangle +Ric_{\tg}\left(\xi,\left(\frac{\partial}{\partial t}\right)^{T}\right)\\
=&\left\langle\frac{\partial}{\partial t},\xi\right\rangle\left(R_{g}+n(n-1)\kappa^{2}-(n-1)(n-2)q^{2}\lambda^{2-2n}\frac{1+f^{2}u_{r}^{2}}{1+f^{2}|\bn u|^{2}}\right)\\
\geq&\left\langle\frac{\partial}{\partial t},\xi\right\rangle\left(R_{g}-R_{\barg}\right),
\end{aligned}
\end{equation}
where $R_{\barg}=-n(n-1)\kappa^{2}+(n-1)(n-2)q^{2}\lambda^{2-2n}$ is the scalar curvature of the  Reissner-Nordstr\"om-AdS manifold.

To caculate the integration on $\Sigma$, we use the assumption that $\Sigma$ is in a level set of $u$ and $|\bn u|\rightarrow\infty$ as $x\rightarrow\Sigma$. Let $\{y^{\alpha}\}_{\alpha=0}^{n-1}$ be a new coordinate system in a neighborhood of $\text{graph}|_{\Sigma}u$ in $M$ , such that $\frac{\partial}{\partial y^{0}}$ is the unit outer normal direction of $\Sigma$ and $\{\frac{\partial}{\partial y^{1}},\cdots,\frac{\partial}{\partial y^{n-1}}\}$ span the tangent space of $\Sigma$. Then on $\Sigma$
$$\partial_{a}u=0,\ |\partial_{0}u|=|\bn u|=+\infty,\ 1\leq a\leq n-1$$
We set $e_{a}=u_{a}\partial_t+\frac{\partial}{\partial y^{a}}=\frac{\partial}{\partial y^{a}},\ e_{0}=u_{0}\partial_t+\frac{\partial}{\partial y^{0}}$ be the corresponding coordinate frame on $M$. Then
\begin{equation*}
\nu_{\Sigma}=\frac{e_{0}}{|e_{0}|},\ |e_{0}|=\sqrt{f^{2}|u_{0}|^{2}+1},\ d\mu_{\Sigma}=d\mu,\ g_{ab}=\barg_{ab}
\end{equation*}
where $d\mu$ is the area element of $\Sigma$ with respect to $\barg$.\\
From \eqref{sff-M}, we get on $\Sigma$,
\begin{equation}\label{graph Sigma II}
B_{ab}=\frac{fu_{0}}{\sqrt{1+f^{2}|\bn u|^{2}}}h_{ab}^{\Sigma},\qquad 0<a,b<n.
\end{equation}
where $\Gamma_{\alpha\beta}^{0}$ is the Christoffel symbols of $(P,\barg)$ with respect to  the coordinate $\{y^{\alpha}\}_{\alpha=0}^{n-1}$ and  $h_{ab}^{\Sigma}$ is the second fundamental form of $\Sigma$ in $(P,\barg)$. From \eqref{graph Sigma II} and $\lim_{x\rightarrow \Sigma}\frac{f|\bn u|}{\sqrt{1+f^{2}|\bn u|^{2}}}=1$, we infer
\begin{equation}
\begin{aligned}
\left\langle T_{1}\left(\frac{\partial}{\partial t}\right)^{T},\nu_{\Sigma}\right\rangle_{g}
=&\left(trB-B_{0}^{0}\right)\left\langle\frac{f^{2}u_{0}}{1+f^{2}|\bn u|^{2}} e_{0},\frac{e_{0}}{|e_{0}|}\right\rangle\\
=&\frac{fu_{0}}{\sqrt{1+f^{2}|\bn u|^{2}}}H\frac{f^{2}u_{0}}{1+f^{2}|\bn u|^{2}}|e_{0}|\\
=&Hf,
\end{aligned}
\end{equation}
which yields
\begin{equation}\label{integrate on Sigma}
\int_{\Sigma}\left\langle T_{1}\left(\frac{\partial}{\partial t}\right)^{T},\nu_{\Sigma}\right\rangle_{g}d\mu_{\Sigma}=\int_{\Sigma}Hfd\mu.
\end{equation}
Finally, combining \eqref{integrate by parts}, \eqref{div} and  \eqref{integrate on Sigma} together, we get
\begin{equation}
m(M,g)\geq m+c_n\left(\int_{M}\left(R_{g}-R_{\barg}\right)\left\langle\frac{\partial}{\partial t},\xi\right\rangle d\mu_{M}+\int_{\Sigma}Hfd\mu\right).
\end{equation}
From \eqref{div} we know equality holds if and only if $u_{0}^{2}=|\bn u|^{2}$, which means that $M$ is rotationaly symmetric.
\end{proof}

Finally, we can prove the Penrose type inequality for ALH graphs.
\begin{proof}[\textbf{Proof of Theorem \ref{Penrose}}]
	It's easy to check that on the horizon $\partial P$,
	\begin{equation}\label{horizon mass}
	2m=\epsilon\left(\frac{|\partial P|}{\vartheta_{n-1}}\right)^{\frac{n-2}{n-1}}+\kappa^{2}\left(\frac{|\partial P|}{\vartheta_{n-1}}\right)^{\frac{n}{n-1}}+q^{2}\left(\frac{|\partial P|}{\vartheta_{n-1}}\right)^{\frac{-n+2}{n-1}}.
	\end{equation}
	Combining \eqref{horizon mass}, \eqref{Penrose0-ineq1} with Theorem \ref{AF-ineq}, we obtain the Penrose type inequality \eqref{Penrose-ineq1}.
\end{proof}

\noindent {\it Acknowledgement:} The  authors wish to thank Dr. Yong Wei for some valuable discussions and helpful comments.

\providecommand{\bysame}{\leavevmode\hbox
	to3em{\hrulefill}\thinspace}

\begin{flushleft}
	Daguang Chen,
	E-mail: dgchen@math.tsinghua.edu.cn\\
	
	Haizhong Li,
	E-mail:	hli@math.tsinghua.edu.cn\\
	
	Tailong Zhou,
	E-mail:	zhou-tl14@mails.tsinghua.edu.cn\\

	Department of Mathematical Sciences, Tsinghua University, Beijing, 100084, P.R.China \\	
	
\end{flushleft}

\end{document}